\documentclass{cimart}

\usepackage{enumitem}
\usepackage{tikz}
\usetikzlibrary{arrows}

\newtheorem{thm}{Theorem}[section]
\newtheorem{cor}[thm]{Corollary}
\newtheorem{lem}[thm]{Lemma}
\newtheorem{prop}[thm]{Proposition}
\newtheorem{defn}[thm]{Definition}
\newtheorem{rem}[thm]{Remark}
\newtheorem{exmp}[thm]{Example}
\newtheorem{conj}{Conjecture}

\newtheorem{thmA}{Theorem}

\newtheorem{prob}{Problem}

\makeatletter
\renewcommand*{\p@section}{\S\,}
\renewcommand*{\p@subsection}{\S\,}
\renewcommand*{\p@subsubsection}{\S\,}
\makeatother

\newcommand{\N}{\ensuremath{\mathbb{N}}}

\newcommand{\K}{\ensuremath{\mathbb{K}}}
\newcommand{\tr}{\operatorname{tr}}

\newcommand{\End}{\operatorname{End}}
\newcommand{\Mat}{\operatorname{Mat}}
\newcommand{\Rep}{\operatorname{Rep}}
\newcommand{\GL}{\operatorname{GL}}
\newcommand{\mult}{\operatorname{m}}

\newcommand{\DJac}{\mathbb{D}\!\operatorname{Jac}}

\newcommand\br[1]{\{ #1 \}}
\newcommand\dhh[1]{\{ #1 \}_{H_0}}
\newcommand\dgal[1]{\{\!\!\{ #1 \}\!\!\}}

\title[Modified double brackets and a conjecture of S.~Arthamonov]{
    Modified double brackets and \texorpdfstring{\\}{}a conjecture of S.~Arthamonov
    }

\author{
    Maxime Fairon
    }

\authorinfo{University of Burgundy \& CNRS, IMB, France}{maxime.fairon@u-bourgogne.fr}

\abstract{%
Around 20 years ago, M.~Van den Bergh introduced double Poisson brackets as operations on associative algebras inducing Poisson brackets under the representation functor.
Weaker versions of these operations, called modified double Poisson brackets, were later introduced by S.~Arthamonov in order to induce a Poisson bracket on moduli spaces of representations of the corresponding associative algebras. Moreover, he defined two operations that he conjectured to be modified double Poisson brackets.
The first case of this conjecture was recently proved by M.~Goncharov and V.~Gubarev motivated by the theory of Rota-Baxter operators of nonzero weight. We settle the conjecture by realising the second case as part of a new family of modified double Poisson brackets.
These are obtained from mixed double Poisson algebras, a new class of algebraic structures that are introduced and studied in the present work.
    }



\VOLUME{33}
\YEAR{2025}
\ISSUE{3}
\NUMBER{5}
\DOI{https://doi.org/10.46298/cm.13786}

\begin{document}




\section{Introduction}

A guiding principle for developing noncommutative algebraic geometry was formulated by Kontsevich and Rosenberg \cite{KR00}. Their idea consists in introducing new structures on associative algebras such that, under each representation functor
\begin{equation*}
 \Rep_N : \mathtt{Ass}_\K \to \mathtt{ComAss}_\K, \quad A \mapsto \K[\Rep(A,N)], \qquad N \in \N,
\end{equation*}
we recover some well-known classical structure.
This principle shaped a facet of noncommutative Poisson geometry following the work of Van den Bergh \cite{VdB1}. Indeed, the notion of \emph{double Poisson brackets} (cf. Definition~\ref{Def:DPA}) on noncommutative algebras was introduced by Van den Bergh to induce a (usual) Poisson bracket on any representation scheme $\Rep(A,N)$.
Interestingly, $\Rep(A,N)$ is naturally equipped with a $\GL_N(\K)$ action that acts by Poisson automorphisms with respect to the Poisson structure induced by a double Poisson bracket. Hence, Van den Bergh's theory also induces a Poisson bracket on the moduli space of representations
$\Rep(A,N)/\!/ \GL_N(\K)$.
This led to a second direction of research, where one wants to induce a Poisson bracket directly on $\Rep(A,N)/\!/ \GL_N(\K)$ that may not have any special property on $\Rep(A,N)$. The weakest such instance is given by the \emph{$H_0$-Poisson structures} of Crawley-Boevey \cite{CB11} (cf. Definition~\ref{Def:H0}).
Another instance is provided by Arthamonov's \emph{modified double Poisson brackets} \cite{Art15,Art17} (cf. Definition~\ref{Def:MDPA}). The latter have the advantage of being ``computable'' since the operation enjoys derivation rules and, therefore, it only needs to be defined on generators of $A$. However, extra axioms are difficult to verify and only a single example could be fully treated \cite[\S3.4]{Art15}.
Two additional examples were conjectured to exist, as the following shows.

\begin{conj}[\hspace{-0.001cm}\cite{Art17}] \label{Conj:A}
 On $A=\K\langle x_1,x_2,x_3\rangle$, the following operations define modified double Poisson brackets:
\begin{equation} \label{Eq:MDBI}
 \begin{aligned}
&\dgal{x_1,x_2}^I=-x_2x_1 \otimes 1,  \qquad &&\dgal{x_2,x_1}^I=x_1x_2 \otimes 1, \\
&\dgal{x_2,x_3}^I=-x_2\otimes x_3,  \qquad &&\dgal{x_3,x_2}^I=x_2 \otimes x_3, \\
&\dgal{x_3,x_1}^I=-1 \otimes x_3x_1,  \qquad &&\dgal{x_1,x_3}^I=1\otimes x_1x_3,
 \end{aligned}
\end{equation}
 and
\begin{equation} \label{Eq:MDBII}
 \begin{aligned}
&\dgal{x_1,x_2}^{I\!I}=-x_1 \otimes x_2,  \quad &&\dgal{x_2,x_1}^{I\!I}=x_1 \otimes x_2, \\
&\dgal{x_2,x_3}^{I\!I}=x_3\otimes x_2,  \quad &&\dgal{x_3,x_2}^{I\!I}=-x_3 \otimes x_2, \\
&\dgal{x_3,x_1}^{I\!I}=x_1 \otimes x_3 - x_3 \otimes x_1,&&
 \end{aligned}
\end{equation}
where the remaining omitted terms involving pairs of generators are assumed to be zero.
\end{conj}

The motivation for this manuscript is to establish the following.

\begin{thmA} \label{Thm:Conj}
 Conjecture \ref{Conj:A} holds true.
\end{thmA}
\begin{proof}
 The case $\dgal{-,-}^{I\!I}$ goes back to Gubarev and Goncharov \cite{GG22}, see Theorem \ref{Thm:GGconj} or \ref{ss:ArtConj} for an independent proof.
The case $\dgal{-,-}^{I}$ is treated in  Theorem \ref{Thm:AI}.
\end{proof}

After the first version of this manuscript appeared on arXiv, we were informed by Vsevolod Gubarev that Andrey Savel'ev independently proved the result under his supervision at Novosibirsk State University~\cite{Sav}. Hence, we should emphasize that the present manuscript settles the conjecture as part of a general classification, not a study of the particular case $\dgal{-,-}^{I}$.
To achieve such a classification, we are led to introduce a new class of algebraic structures, defined as follows. Consider the free $\K$-algebra $A=\K\langle v_1,\ldots,v_d\rangle$ equipped with a linear map $\dgal{-,-}:A\otimes A \to A\otimes A$ satisfying the Leibniz rules \eqref{Eq:Lei}.
Assume that there exists $\underline{\lambda}=(\lambda_1,\ldots,\lambda_d)\in \K^d$ such that the mapping $\dgal{-,-}$ satisfies for any $1\leq i,j\leq d$,
\begin{equation} \label{Eq:intr1}
 \dgal{v_i,v_j} + \dgal{v_j , v_i}^\circ
 = \frac{\lambda_i+\lambda_j}{2} (v_i \otimes v_j - v_j \otimes v_i)
+ \frac{\lambda_i-\lambda_j}{2} \, (1\otimes v_i v_j - \, v_j v_i \otimes 1)\,,
\end{equation}
together with the \emph{Poisson property} given by \eqref{Eq:DJac-v}.
The pair $(A,\dgal{-,-})$ is called a \emph{mixed double Poisson algebra}.
Our main result is the following.
\begin{thmA} \label{Thm:MAIN}
If $(A,\dgal{-,-})$ is a mixed double Poisson algebra of weight $\underline{\lambda}$,
then $\dgal{-,-}$ is a modified double Poisson bracket.
\end{thmA}

\begin{proof}
The $3$ conditions of Definition~\ref{Def:MDPA} of a modified double Poisson bracket (viz. Leibniz rules, $H_0$-skew-symmetry and Jacobi identity) are satisfied by definition and Corollary \ref{Cor:skew} together with Proposition \ref{Pr:Jac}.
\end{proof}
The importance of Theorem \ref{Thm:MAIN} is to reduce the complicated task of showing that an operation is a modified double Poisson bracket to checking a finite number of identities given by \eqref{Eq:intr1} and \eqref{Eq:DJac-v}.
In particular, we shall deduce Theorem \ref{Thm:Conj} by showing that both of Arthamonov's operations \eqref{Eq:MDBI} and \eqref{Eq:MDBII} are mixed double Poisson algebras.

\subsection*{Layout}
In \ref{S:NCPoiss}, we gather the necessary definitions and properties to  make Conjecture~\ref{Conj:A} precise.
Of importance, we include the recent notion of $\lambda$-double Lie algebras of Goncharov-Gubarev \cite{GG22}, which provided a framework to prove the second case of Conjecture \ref{Conj:A}.
In \ref{S:MDA} and \ref{S:Jacobi}, we work towards generalising the approach of Goncharov-Gubarev.
We introduce mixed double Poisson algebras of some weight valued in $\K^d$, and show that the operation that equips these algebras is a modified double Poisson bracket, thus establishing Theorem~\ref{Thm:MAIN}.
In \ref{S:Loc}, we prove that one can extend a mixed double Poisson algebra structure on a free algebra to the corresponding algebra of noncommutative Laurent polynomials.
Finally, we present several families of examples in \ref{S:ClassConj} (based on a classification given in the appendix), which allow us to deduce the validity of Conjecture \ref{Conj:A}.

\subsection*{Notation} Throughout the manuscript, $\K$ is an algebraically closed field of characteristic zero. Algebras are finitely generated associative unital algebras over $\K$. Unadorned tensor products are over $\K$.
We only work with $\K$-linear maps, and therefore we shall denote an element $u\in A\otimes A$
using the Sweedler-type notation $u'\otimes u''$, even if it not a pure tensor.

\section{Noncommutative Poisson geometry} \label{S:NCPoiss}

\subsection{The approaches of Van den Bergh, Crawley-Boevey and Arthamonov} \label{ss:VdBCBA}

Let us fix an algebra $A$ and endow $A\otimes A$ with its natural multiplication given by $(a'\otimes a'')(b'\otimes b'')=a'b' \otimes a'',b''$, where $a',a'',b',b''\in A$.
We focus on linear operations of the form
\begin{equation}
 \dgal{-,-}:A\otimes A \to A\otimes A, \quad a\otimes b \mapsto \dgal{a,b}.
\end{equation}
(It is equivalent to bilinear maps with domain $A\times A$, as the notation suggests).
These can be extended to operations $A^{\otimes 3} \to A^{\otimes 3}$ as follows
\begin{subequations} \label{Eq:dbr3}
 \begin{align}
\dgal{a,b\otimes c}_L &= \dgal{a,b} \otimes c, \\
\dgal{a,b\otimes c}_R &= b\otimes  \dgal{a,c} , \\
\dgal{b\otimes c,a}_L &= \dgal{b,a} \otimes_1 c,
 \end{align}
\end{subequations}
for any $a,b,c\in A$, where we write
$(a\otimes b) \otimes_1 c = a\otimes c \otimes b = c \otimes_1 (a\otimes b)$.
Then, we define the \emph{double Jacobiator} $\DJac:A^{\otimes 3} \to A^{\otimes 3}$ by
\begin{equation} \label{Eq:DJac}
 \DJac(a,b,c)=\dgal{a,\dgal{b,c}}_L - \dgal{b,\dgal{a,c}}_R - \dgal{\dgal{a,b},c}_L\,.
\end{equation}
Van den Bergh's notion of double Poisson brackets \cite{VdB1} is given as follows.

\begin{defn} \label{Def:DPA}
A linear map $\dgal{-,-}:A\otimes A \to A\otimes A$ satisfying for any $a,b,c\in A$ the Leibniz rules
\begin{subequations} \label{Eq:Lei}
 \begin{align}
  \dgal{a,bc}=(b\otimes 1) \dgal{a,c} + \dgal{a,b} (1\otimes c)\,,   \label{Eq:Lei1} \\
  \dgal{ab,c}=(1\otimes a) \dgal{b,c} + \dgal{a,c} (b\otimes 1)\,,   \label{Eq:Lei2}
 \end{align}
\end{subequations}
is called a \emph{double bracket} if the following cyclic skew-symmetry rule holds:
\begin{equation} \label{Eq:cskew}
 \dgal{a,b} =- \dgal{b,a}^\circ\,, \qquad a,b\in A\,.
\end{equation}
When a double bracket has vanishing double Jacobiator, i.e. $\DJac(a,b,c)=0$ for \eqref{Eq:DJac},
we call it a \emph{double Poisson bracket}.
\end{defn}

\begin{rem}
The above definition is closely related to that given by De Sole, Kac and Valeri \cite{DSKV}, which reformulates \cite[\S\S2.2-2.3]{VdB1}.
 The Leibniz rules, according to Van den Bergh, are stated in terms of the $A$-bimodule structures on $A\otimes A$:
\begin{subequations} \label{Eq:Bimod}
  \begin{align}
c_1\cdot (a\otimes b) \cdot c_2&= (c_1\otimes 1)(a\otimes b)(1\otimes c_2)\,, \quad \text{\it (outer bimodule)} \\
c_1\ast (a\otimes b) \ast c_2&= (1\otimes c_1)(a\otimes b)(c_2\otimes 1)\,. \quad \text{\it (inner bimodule)}
 \end{align}
\end{subequations}
Due to cyclic skew-symmetry \eqref{Eq:cskew}, the form of $\DJac$ \eqref{Eq:DJac} is equivalent to Van den Bergh's original triple bracket, see \cite[Rem.~2.2]{DSKV}.
\end{rem}

Let us point out the following useful property.

\begin{lem} \label{Lem:DJac}
If the Leibniz rules \eqref{Eq:Lei} hold,
the operation $\DJac:A^{\otimes 3}\to A^{\otimes 3}$ is a derivation in the second and third arguments as follows:
\begin{equation}
\begin{aligned}
 \DJac(a,b,c_1c_2)&= (c_1\otimes 1\otimes 1) \DJac(a,b,c_2) + \DJac(a,b,c_1) (1\otimes 1\otimes c_2)\,, \\
 \DJac(a,b_1b_2,c)&= (1\otimes 1\otimes b_1) \DJac(a,b_2,c) + \DJac(a,b_1,c) (1\otimes b_2 \otimes 1)\,,
\end{aligned}
\end{equation}
for the multiplication $(a'\otimes a'' \otimes a''')(b'\otimes b'' \otimes b''')=a'b'\otimes a''b'' \otimes a'''b'''$ in $A^{\otimes 3}$.
Moreover,
\begin{equation}
 \begin{aligned} \label{Eq:DJac-DerNot}
 \DJac(a_1a_2,b,c)&= (1\otimes a_1\otimes 1) \DJac(a_2,b,c) + \DJac(a_1,b,c) (a_2\otimes 1 \otimes 1) \\
 &-\dgal{a_2,c}' \otimes (\dgal{b,a_1} + \dgal{a_1,b}^\circ) \dgal{a_2,c}''\,.
\end{aligned}
\end{equation}
(Recall the Sweedler-type notation $\dgal{a_2,c}=\dgal{a_2,c}'\otimes \dgal{a_2,c}''$).
Hence $\DJac$ is a derivation in the first argument only when cyclic skew-symmetry \eqref{Eq:cskew} holds.
\end{lem}
\begin{proof}
This is a direct computation.
For the first identity, cf. \cite[(3.11)]{DSKV} with $\lambda=\mu=0$. The second case is similar.
For \eqref{Eq:DJac-DerNot}, we refer to \cite[\S7.1]{Art17}.
\end{proof}

A prominent feature of double Poisson algebras is the following result.

\begin{thm}[\hspace{-0.001cm}\cite{VdB1}, \S7.5] \label{Thm:VdB}
 If $\dgal{-,-}$ is a double Poisson bracket on $A$, it induces a Poisson bracket on the $N$-th representation space $\Rep(A,N)$. Furthermore, the natural action of $\GL_N(\K)$ on $\Rep(A,N)$ is by Poisson automorphisms, and the induced Poisson bracket descends to the GIT quotient $\Rep(A,N)/\!/ \GL_N(\K)$.
\end{thm}

In parallel\footnote{While published in 2011 \cite{CB11}, a preprint containing these ideas under the name of ``noncommutative Poisson structures'' appeared on the arXiv in 2005. We also draw the attention of the reader to the related notion of Hamiltonian operators on free algebras by Mikhailov and Sokolov \cite{MS00}.} to Van den Bergh's work,
Crawley-Boevey introduced the notion of $H_0$-Poisson structures \cite{CB11}, which are a family of weaker structures inducing a Poisson bracket on the moduli space $\Rep(A,N)/\!/ \GL_N(\K)$ (see \cite[Thm.~1.6]{CB11} for a precise statement).
Let $[A,A]$ be the vector space of commutators in $A$, and set $H_0(A):=A/[A,A]$, the zero-th Hochschild homology of $A$. Denote by $A \to H_0(A)$, $a \mapsto \bar{a}$, the corresponding linear quotient map. Remark that any derivation $\delta$ on $A$ induces a linear map on $H_0(A)$.
\begin{defn}[\hspace{0.001cm}\cite{CB11}] \label{Def:H0}
A linear map $\dhh{-,-}: H_0(A)\otimes H_0(A)\to H_0(A)$ is a \emph{$H_0$-Poisson structure} on $A$  if it is a Lie bracket on $H_0(A)$
such that for all $a\in A$,  the linear map $\dhh{\bar a,-} : H_0(A)\to H_0(A)$ is induced by a derivation on $A$.
\end{defn}

Any double Poisson bracket leads to a $H_0$-Poisson structure through the composition $\mult \circ \dgal{-,-}$ with the multiplication map $\mult : A\otimes A \to A$, before restricting to $H_0(A)$, cf. \cite[Lem.~2.6.2]{VdB1}.
The converse is far from being true, as $H_0$-Poisson structures are much more general. Nevertheless, Crawley-Boevey's notion has a big shortcoming: it is difficult to construct or characterise $H_0$-Poisson structure.
Arthamonov attempted to rectify this problem by building a class of ``computable'' $H_0$-Poisson structures.

\begin{defn}[\hspace{0.001cm}\cite{Art15,Art17}] \label{Def:MDPA}
A linear map $\dgal{-,-}:A\otimes A \to A\otimes A$ satisfying the Leibniz rules \eqref{Eq:Lei} for any $a,b,c\in A$ is called a \emph{modified double bracket} if the following $H_0$-skew-symmetry rule holds:
\begin{equation} \label{Eq:Skew}
 \{a,b\}+\{b,a\} \in [A,A]\,,
\end{equation}
where $\{-,-\}= \mult \circ \dgal{-,-}$ for the multiplication $\mult : A\otimes A \to A$ on $A$.
When a modified double bracket satisfies the Jacobi identity
\begin{equation} \label{Eq:Jac}
 \{a,\{b,c\}\}-\{b,\{a,c\}\} -\{\{a,b\},c\}=0\,,
\end{equation}
we call it a \emph{modified double Poisson bracket}.
\end{defn}

It is clear that a modified double Poisson bracket induces a $H_0$-Poisson structure by restriction of $\{-,-\}$ to $H_0(A)$, in analogy with the case of double Poisson brackets.
Our previous discussions entail the following generalisation of Theorem \ref{Thm:VdB}.

\begin{thm}[\hspace{0.001cm}\cite{Art17}, \S3 \& \cite{CB11}, \S4]
Any $H_0$-Poisson structure on $A$ uniquely induces a Poisson bracket on the GIT quotient $\Rep(A,N)/\!/ \GL_N(\K)$.
In particular, any modified double Poisson bracket on $A$ uniquely induces a Poisson bracket on $\Rep(A,N)/\!/ \GL_N(\K)$.
\end{thm}

A modified double (Poisson) bracket has the advantage of only requiring to be defined on generators due to the Leibniz rules.
Yet again, such structures are challenging to find because it is not simple to verify the rules \eqref{Eq:Skew} and \eqref{Eq:Jac}. In fact, if we exclude Van den Bergh's double Poisson brackets (easily seen to satisfy Definition~\ref{Def:MDPA}), Arthamonov only managed to fully check the axioms of a modified double Poisson bracket in a single case \cite[\S3.4]{Art15}, and he conjectured the existence of the two extras cases featured in Conjecture~\ref{Conj:A}.
Therefore, an important open problem consists in building new examples of modified double Poisson brackets. A breakthrough in that direction has recently occurred \cite{GG22}, which we describe in the next subsection.

\subsection{The approach of Goncharov-Gubarev}

The manuscript \cite{GG22} is based on the following observations.
Firstly, for $V:=\oplus_{i=1}^d \K v_j$, consider a double Poisson bracket on the free algebra $\operatorname{Ass}(V)$ generated by $V$ such that $\dgal{v_i,v_j} \in V\otimes V$ for all $1\leq i,j \leq d$.
Given dual bases $(e_k)$ and $(e^k)$ of $\End(V)\simeq \Mat_{d}(\K)$ under the trace pairing, define
\begin{equation} \label{Eq:RBop}
 R: \End(V) \to \End(V)
\end{equation}
uniquely through the following decomposition:
\begin{equation} \label{Eq:RB-br}
 \dgal{u,v}=\sum_{1\leq k \leq d^2} e_k(u) \otimes R(e^k)(v)\,, \quad u,v\in V\,.
\end{equation}
One can check \cite{GK18} that the operation $R$ \eqref{Eq:RBop} hence obtained is a skew-symmetric Rota-Baxter operator on $\End(V)$, meaning that $R=-R^\ast$ (dual for the trace pairing) and
\begin{equation} \label{Eq:RB-0}
 R(e)R(f) = R(R(e)f + eR(f) )\,, \quad e,f\in \End(V).
\end{equation}
Secondly, the theory of Rota-Baxter operators extends to \emph{non-zero weight} $\lambda\in \K^\times$, where the right-hand side of \eqref{Eq:RB-0} contains the extra term $+\lambda R(ef)$. Using this generalised notion and the $\lambda$-skew-symmetry $R=-R^\ast+\lambda \tr(-) 1_{\End(V)}-\lambda 1_{\End(V)}(-)$,
Goncharov and Gubarev \cite{GG22} made the following definition still based on \eqref{Eq:RB-br}.

\begin{defn}[\hspace{0.001cm}\cite{GG22}, Def.~4]  \label{Def:GG}
A  $\lambda$-double Lie algebra structure on a vector space $V$ is a
linear map $\dgal{-,-}:V\otimes V \to V\otimes V$ such that for any $u,v,w \in V$,
\begin{subequations}
  \begin{align}
\dgal{u,v}+\dgal{v,u}^\circ&=\lambda\, (u \otimes v - v \otimes u)\,, \label{Eq:GG-sk}\\
 \DJac(u,v,w)&=-\lambda\, v\otimes_1 \dgal{u,w}\,. \label{Eq:GG-Jac}
\end{align}
\end{subequations}
\end{defn}

\begin{rem}
We stress that a $\lambda$-double Lie algebra is not endowed with an associative multiplication compatible with $\dgal{-,-}$, as opposed to the (modified) double Poisson brackets as in \ref{ss:VdBCBA}.
 Up to rescaling, there are two unequivalent cases: $\lambda=0$ and $\lambda=1$.
 The $\lambda=0$ case corresponds to a double Lie algebra as introduced e.g. in \cite{DSKV}.
\end{rem}

The main results of Goncharov and Gubarev are the following.

\begin{thm}[\hspace{0.001cm}\cite{GG22}, Thm.~10]   \label{Thm:GG}
Consider a $\lambda$-double Lie algebra structure $\dgal{-,-}$ on a finite-dimensional vector space $V$.
Then its extension to $A=\operatorname{Ass}(V)$ through the Leibniz rules \eqref{Eq:Lei} is a modified double Poisson bracket.
\end{thm}

\begin{thm}[\hspace{0.001cm}\cite{GG22}, Cor.~4]   \label{Thm:GGconj}
The operation $\dgal{-,-}^{I\!I}$ \eqref{Eq:MDBII} on  $A=\K\langle x_1,x_2,x_3\rangle$ is a modified double Poisson bracket.
\end{thm}

In particular, this settled the second case of Conjecture~\ref{Conj:A}.
Forgetting about the conjecture, Goncharov and Gubarev's work allowed to define a new class of structures that sits between double Poisson brackets and modified double Poisson brackets.
The present work aims at introducing and studying the new notion of \emph{mixed double Poisson algebras of weight}
$\underline{\lambda}=(\lambda_{1},\ldots,\lambda_{d}) \in \K^d$ (cf. Definition~\ref{Def:PwDA}), which is related to the previous structures as follows when
$A:=\operatorname{Ass}(V)$ for $V=\oplus_{i=1}^d \K v_j$:
\begin{center}
\begin{tikzpicture}
    \node (A) at (0,3) {$\{$extended double Lie algebras$\}$};
    \node (B1) at (-4,1.5) {$\{$double Poisson algebras$\}$};
    \node (B2) at (4,1.5) {$\{$extended $\lambda$-double Lie algebras$\}$};
    \node (C) at (0,0) {$\{$mixed double Poisson algebras of weight $\underline{\lambda} \}$};
    \node (D) at (0,-1.5) {$\{$modified double Poisson algebras$\}$};
 \draw[left hook->, gray, thick,>=angle 90] (A) -- (B1);
 \node[gray,font=\scriptsize] (AB1) at (-3.5,2.4) {restricts to $V^{\otimes 2} \!\to\! V^{\otimes 2}$};
 \draw[right hook->, gray, thick,>=angle 90] (A) -- (B2);
 \node[gray,font=\scriptsize] (AB2) at (3,2.4) {case $\lambda\!=\!0$};
 \draw[right hook->, gray, thick,>=angle 90] (B1) -- (C);
 \node[gray,font=\scriptsize] (B1C) at (-3.8,0.8) {case $\underline{\lambda}\!=\!(0,\ldots,0)$};
 \draw[left hook->, gray, thick,>=angle 90] (B2) -- (C);
 \node[gray,font=\scriptsize] (B2C) at (4.4,0.95) {restricts to $V^{\otimes 2} \!\to\! V^{\otimes 2}$};
 \node[gray,font=\scriptsize] (B2Cbis) at (4.1,0.6) {with $\underline{\lambda}\!=\!(\lambda,\ldots,\lambda)$};
 \draw[left hook->, gray, thick,>=angle 90] (C) -- (D);
 \node[gray,font=\scriptsize] (CD) at (-0.1,-0.8) {special \,\, cases};
\end{tikzpicture}
\end{center}

The inclusion on the right of the second line will be explained in Example \ref{Exmp:GG}.
The upshot is that all these families provide examples of the weakest structure : $H_0$-Poisson structures.
Indeed, this is a consequence of Theorem \ref{Thm:MAIN}, which can be seen as a generalisation of Theorem \ref{Thm:GG}.


\section{Mixed double algebras} \label{S:MDA}

For $d\geq 1$, we let $A=\K\langle v_1,\ldots,v_d\rangle$.
We aim at generalising the skew-symmetry rule \eqref{Eq:GG-sk} of Goncharov and Gubarev.

\subsection{First definition} \label{ss:MDA1}

We fix two matrices
$\Lambda,M \in \Mat_d(\K)$ where $\Lambda$ is symmetric, while $M$ is skew-symmetric.
In terms of the entries $(\lambda_{ij})_{i,j=1}^d$ and $(\mu_{ij})_{i,j=1}^d$ of $\Lambda$ and $M$, this means that
for $1\leq i,j\leq d$ with $i\neq j$:
\begin{equation} \label{Eq:CondLM}
 \lambda_{ij}=\lambda_{ji}; \quad \mu_{ij}=-\mu_{ji}, \ \ \mu_{ii}=0\,.
\end{equation}

\begin{defn}  \label{Def:MixDA}
 Given a linear map $\dgal{-,-}:A\otimes A \to A\otimes A$ satisfying the Leibniz rules \eqref{Eq:Lei}, we say that the pair $(A,\dgal{-,-})$ is a \emph{mixed double algebra of type} $(\Lambda,M)$ if, for any $1\leq i,j\leq d$,
\begin{equation} \label{Eq:mixDA}
 \dgal{v_i,v_j} + \dgal{v_j , v_i}^\circ
 = \lambda_{ij} (v_i \otimes v_j - v_j \otimes v_i)
+ \mu_{ij} \, 1\otimes v_i v_j + \mu_{ji} \, v_j v_i \otimes 1\,.
\end{equation}
\end{defn}

We make the following observations:
\begin{enumerate}
 \item The condition \eqref{Eq:mixDA} is well-defined since it is preserved by applying the permutation of tensor factors $(-)^\circ$, which amounts to swapping the indices $i,j$.
 \item The type $(\Lambda,M)$ of a mixed double algebra depends on the chosen presentation of the free algebra. E.g. permuting generators amounts to conjugating the pair  $(\Lambda,M)$ by the corresponding permutation matrix.
 Furthermore, we can always replace $\Lambda$ by $\Lambda+D$, for $D$ a diagonal matrix.
 \item Multiplying $\dgal{-,-}$ by a factor $\nu \in \K$ changes the type to $(\nu\Lambda,\nu M)$.
This is analogous to the fact that we can multiply (modified) double (Poisson) brackets, and that a $\lambda$-double Lie algebra becomes a $\nu \lambda$-double Lie algebra.
 \item Taking $i=j$ in \eqref{Eq:mixDA} yields the cyclic skew-symmetry $\dgal{v_i,v_i} =- \dgal{v_i , v_i}^\circ$. Thus the case $d=1$ just restricts to Van den Bergh's definition of a double bracket \cite{VdB1}, and we will assume that $d\geq 2$ hereafter.
 \item If $M=0_d$ and if all entries of $\Lambda$ are equal to a fixed $\lambda\in \K$, we recover the condition \eqref{Eq:GG-sk} of Goncharov-Gubarev~\cite{GG22}.
\end{enumerate}

Let us introduce some convenient notation.
Since $A=\K\langle v_1,\ldots,v_d\rangle$, any element can be written as a constant term $\nu \in \K$ added to a linear combination of terms of the form
\begin{equation*}
 a=v_{i_1} \cdots v_{i_r}, \quad \text{where }i_1,\ldots,i_r \in \{1,\ldots,d\}, \ r \geq 1\,.
\end{equation*}
For such a term $a\in A$, we set for any $1\leq \alpha,\gamma \leq r$,
\begin{equation} \label{Eq:Not}
 a_\alpha^-:=v_{i_1} \cdots v_{i_{\alpha-1}}, \quad
 a_\alpha^+:=v_{i_{\alpha+1}} \cdots v_{i_r}, \quad
 a_{\alpha,\gamma}^\sim :=\left\{
\begin{array}{ll}
v_{i_{\alpha}} \cdots v_{i_\gamma}, & \alpha \leq \gamma, \\
1,& \alpha>\gamma,
\end{array}
 \right.
\end{equation}
so that $a=a_\alpha^- v_{i_{\alpha}}  a_\alpha^+$ and $a=a_\alpha^- a_{\alpha,\gamma}^\sim  a_\gamma^+$ if $\alpha\leq \gamma$.

\begin{prop} \label{Pr:wsk}
 Let $(A,\dgal{-,-})$ be a mixed double algebra of type $(\Lambda,M)$.
If the following conditions are satisfied
\begin{equation}
\lambda_{ij}-\lambda_{kl}=\mu_{il}-\mu_{kj}\,, \quad  1\leq i,j,k,l \leq d,   \label{Eq:wskM}
\end{equation}
then the $H_0$-skew-symmetry rule \eqref{Eq:Skew} holds for any $a,b\in A$.
\end{prop}

\begin{proof}
 By linearity, it suffices to verify \eqref{Eq:Skew} for $a,b\in A$ of the form
 $a=v_{i_1} \cdots v_{i_r}$ and $b=v_{j_1} \cdots v_{j_s}$ with indices in $\{1,\ldots,d\}$ and $r,s\geq 1$.
Using the Leibniz rules \eqref{Eq:Lei}, the mixed double algebra condition \eqref{Eq:mixDA} and the notation \eqref{Eq:Not}, we can write
\begin{align*}
 \dgal{a,b}+\dgal{b,a}^\circ
&=\sum_{\alpha=1}^r \sum_{\beta=1}^s \,
(b_\beta^- \otimes a_\alpha^-) (\dgal{v_{i_\alpha},v_{j_\beta}}+\dgal{v_{j_\beta},v_{i_\alpha}}^\circ ) (a_\alpha^+ \otimes b_\beta^+) \\
&= \sum_{\alpha=1}^r \sum_{\beta=1}^s \,\lambda_{i_\alpha j_\beta}(
b_\beta^- v_{i_\alpha} a_\alpha^+  \otimes a_\alpha^- v_{j_\beta} b_\beta^+
- b_\beta^- v_{j_\beta} a_\alpha^+  \otimes a_\alpha^- v_{i_\alpha} b_\beta^+
) \\
&\quad + \sum_{\alpha=1}^r \sum_{\beta=1}^s \,(
\mu_{i_\alpha j_\beta}\, b_\beta^- a_\alpha^+  \otimes a_\alpha^- v_{i_\alpha} v_{j_\beta} b_\beta^+
+\mu_{j_\beta i_\alpha}\, b_\beta^- v_{j_\beta}v_{i_\alpha} a_\alpha^+  \otimes a_\alpha^- b_\beta^+
)\,.
\end{align*}
Multiplying the tensor factors, we get modulo commutators
\begin{align*}
 \{a,b\}+&\{b,a\}^\circ
=\sum_{\alpha=1}^r \sum_{\beta=1}^s \,\lambda_{i_\alpha j_\beta}(
v_{i_\alpha} a_{\alpha}^+  a_\alpha^- v_{j_\beta} b_{\beta}^+ b_\beta^-
- a_\alpha^+ a_{\alpha}^-v_{i_\alpha} b_\beta^+ b_{\beta}^-v_{j_\beta}
) \\
+& \sum_{\alpha=1}^r \sum_{\beta=1}^s \,(
\mu_{i_\alpha j_\beta}\, a_\alpha^+  a_{\alpha}^-v_{i_\alpha} v_{j_\beta}b_{\beta}^+ b_\beta^-
+\mu_{j_\beta i_\alpha}\,  v_{i_\alpha} a_{\alpha}^+ a_\alpha^- b_\beta^+ b_{\beta}^-v_{j_\beta}
)\, \\
=& \sum_{\alpha=1}^r \sum_{\beta=1}^s \,(\lambda_{i_\alpha j_\beta}-\lambda_{i_{\alpha-1} j_{\beta-1}} + \mu_{i_{\alpha-1} j_{\beta}}+\mu_{j_{\beta-1}i_{\alpha}} )
v_{i_\alpha}a_{\alpha}^+  a_\alpha^- v_{j_\beta} b_{\beta}^+ b_\beta^- \quad \text{mod }[A,A].
\end{align*}
Note that we consider indices modulo $r$ or $s$; this is how we used $a_\alpha^+  a_{\alpha}^-v_{i_\alpha}=v_{i_{\alpha+1}}a_{\alpha+1}^+  a_{\alpha+1}^-$ for any $1\leq \alpha \leq r$ and did the same for $b$.
By skew-symmetry of $M$,  this reads
\begin{equation} \label{Eq:Lem1a}
 \sum_{\alpha=1}^r \sum_{\beta=1}^s \,(\lambda_{i_\alpha j_\beta}-\lambda_{i_{\alpha-1} j_{\beta-1}} + \mu_{i_{\alpha-1} j_{\beta}}-\mu_{i_{\alpha}j_{\beta-1}} ) \, \mathrm{V}_{\alpha,\beta}= 0  \quad \text{mod }[A,A],
\end{equation}
for $\mathrm{V}_{\alpha,\beta}:=v_{i_\alpha}a_{\alpha}^+  a_\alpha^- v_{j_\beta} b_{\beta}^+ b_\beta^-$.
Each summand in \eqref{Eq:Lem1a} is identically zero by \eqref{Eq:wskM}.
\end{proof}

\begin{rem}
 It is clear from \eqref{Eq:mixDA} that the choice of diagonal entries $\lambda_{ii}$ can be arbitrary as the first two terms cancel out for $i=j$.
Accordingly, one can prove Proposition \ref{Pr:wsk} under weaker assumptions than \eqref{Eq:wskM}, which omit the cases $i=j$ and $k=l$. We shall not use this more general formalism, and we skip the proof. (The interested reader may find that general case as Lemma 3.1 in v1 of the arXiv version of this manuscript).
\end{rem}

\subsection{Refining the definition} \label{ss:MDA2}

Assume that \eqref{Eq:wskM} is satisfied by $(\Lambda,M)$. We can deduce
\begin{equation} \label{Eq:MuLam}
 \mu_{il}=\frac12(\lambda_{ii}-\lambda_{ll}), \quad \lambda_{il}=\frac12(\lambda_{ii}+\lambda_{ll}), \quad 1\leq i,l \leq d\,.
\end{equation}
Furthermore, given $(\lambda_{11},\ldots,\lambda_{dd})\in \K^d$, we can define matrices $(\Lambda,M)$ through \eqref{Eq:MuLam} and the conditions \eqref{Eq:CondLM},\eqref{Eq:wskM} are automatically satisfied.
We get the next simpler definition.
\begin{defn} \label{Def:weiDA}
Let $\underline{\lambda}=(\lambda_1,\ldots,\lambda_d)\in \K^d$.
 Given a linear map $\dgal{-,-}:A\otimes A \to A\otimes A$ satisfying the Leibniz rules \eqref{Eq:Lei}, we say that the pair $(A,\dgal{-,-})$ is a \emph{mixed double algebra of weight} $\underline{\lambda}$ if \eqref{Eq:intr1} is satisfied for any $1\leq i,j\leq d$.
\end{defn}

We can interpret Proposition \ref{Pr:wsk} as follows.

\begin{cor} \label{Cor:skew}
A mixed double algebra of weight $\underline{\lambda}\in \K^d$ is equipped with a modified double bracket.
\end{cor}

\begin{rem} \label{Rem:IndFree}
 As suggested by a referee, a basis-free version of \eqref{Eq:intr1} can be given as follows.
Define the subspace $V_{\lambda}= \operatorname{span}_\K\{v_i \mid \lambda_i =\lambda, \, 1\leq i \leq d\}$ for any $\lambda\in \K$.
Then, for any $\lambda,\lambda'\in \K$, $x\in V_\lambda$ and $y\in V_{\lambda'}$, we have
\begin{equation*}
  \dgal{x,y} + \dgal{y,x}^\circ
  = \frac{\lambda+\lambda'}{2} \, (x \otimes y - y \otimes x)
 + \frac{\lambda-\lambda'}{2} \, (1\otimes xy - yx \otimes 1)\,.
 \end{equation*}
\end{rem}


\section{Jacobi identity} \label{S:Jacobi}

Fix $d\geq 2$ and $A=\K\langle v_1,\ldots,v_d\rangle$.
Recall Definition \ref{Def:weiDA}.

\begin{defn} \label{Def:PwDA}
A mixed double algebra $(A,\dgal{-,-})$ of weight $\underline{\lambda}\in \K^d$ is \emph{Poisson} when,
for any $1\leq i,j,k \leq d$,
\begin{equation} \label{Eq:DJac-v}
 \DJac(v_i,v_j,v_k)=-\frac{\lambda_i+\lambda_j}{2}\, v_j\otimes_1 \dgal{v_i,v_k}+\frac{\lambda_i-\lambda_j}{2} \, 1\otimes_1 (v_j \ast \dgal{v_i,v_k}).
\end{equation}
We call \eqref{Eq:DJac-v} the \emph{Poisson property} of $\dgal{-,-}$.
\end{defn}

By Lemma \ref{Lem:DJac}, we get in the Poisson case for any $1\leq i,j\leq d$ and $c\in A$,
\begin{equation} \label{Eq:DJac-c}
 \DJac(v_i,v_j,c)=-\frac{\lambda_i+\lambda_j}{2}\, v_j\otimes_1 \dgal{v_i,c}+\frac{\lambda_i-\lambda_j}{2} \, 1\otimes_1 (v_j \ast \dgal{v_i,c})\,.
\end{equation}

\begin{rem}
With the notation of Remark~\ref{Rem:IndFree}, an index-free expression is given for any
$\lambda,\lambda',\mu\in \K$, $x\in V_\lambda$, $y\in V_{\lambda'}$ and $z\in V_\mu$ by
\begin{equation*}
 \DJac(x,y,z)=-\frac{\lambda+\lambda'}{2}\, y\otimes_1 \dgal{x,z}+\frac{\lambda-\lambda'}{2} \, 1\otimes_1 (y \ast \dgal{x,z}).
\end{equation*}
\end{rem}

We shall prove the following result in \ref{ss:PfJac}, as a generalisation of \cite[Thm.~9]{GG22}.

\begin{prop} \label{Pr:Jac}
Fix $(A,\dgal{-,-})$ a mixed double Poisson algebra of weight $\underline{\lambda}\in \K^d$.
Then the Jacobi identity \eqref{Eq:Jac} is satisfied for any $a,b,c\in A$.
\end{prop}

\begin{exmp} \label{Exmp:GG}
Fix $\lambda \in \K$, and let $(A,\dgal{-,-})$ be a mixed double Poisson algebra of weight $(\lambda,\ldots,\lambda)$.
The conditions  \eqref{Eq:intr1} and \eqref{Eq:DJac-v} then read
\begin{align*}
\dgal{v_i,v_j}+ \dgal{v_j,v_i}^\circ&=\lambda\, (v_i \otimes v_j - v_j \otimes v_i)\,, \\
 \DJac(v_i,v_j,v_k)&=-\lambda\, v_j\otimes_1 \dgal{v_i,v_k}\,.
\end{align*}
By linearity, these identities yield \eqref{Eq:GG-sk} and \eqref{Eq:GG-Jac} for any $u,v,w \in V:=\bigoplus_{i=1}^d \K v_i$.
Hence we have that a mixed double Poisson algebra of weight $(\lambda,\ldots,\lambda)$ that restricts to a map
$V\otimes V\to V \otimes V$ (when considered on the $\K$-linear span $V$ of generators) is the extension to $\operatorname{Ass}(V)$ of a $\lambda$-double Lie algebra structure on $V$.
\end{exmp}

\subsection{Preparation}
We  start by simply assuming that $A=\K\langle v_1,\ldots,v_d\rangle$ is equipped with a linear map $\dgal{-,-}:A^{\otimes 2} \to A^{\otimes 2}$ satisfying the Leibniz rules \eqref{Eq:Lei}.
We work with the associated operation $\br{-,-}:= \mult \circ \dgal{-,-}: A\otimes A \to A$. Fix
\begin{equation} \label{Eq:abForm}
 a=v_{i_1} \ldots v_{i_r}, \quad  b=v_{j_1} \ldots v_{j_s}, \quad
 r,s\geq 1, \ \ 1\leq i_1, \ldots, i_r,j_1, \ldots, j_s \leq d\,.
\end{equation}
Recall the notation \eqref{Eq:Not}.
We present formulas that can be found in \cite[pp.24-25]{GG22}.

\begin{lem}
The following holds:
\begin{subequations} \label{Eq:br-A}
 \begin{align}
\br{a,\br{b,c}}=& \
\sum_{\alpha=1}^r \sum_{\beta=1}^s \dgal{v_{i_\alpha} , \dgal{v_{j_\beta},c}'}' a_\alpha^+ a_\alpha^-
\dgal{v_{i_\alpha} , \dgal{v_{j_\beta},c}'}'' b_\beta^+ b_\beta^-  \dgal{v_{j_\beta},c}'' \label{Eq:A1} \\
&+ \sum_{\alpha=1}^r \sum_{\substack{\beta,\epsilon=1 \\ \beta<\epsilon}}^s
\dgal{v_{j_\beta},c}' b_{\beta+1,\epsilon-1}^\sim \dgal{v_{i_\alpha} , v_{j_\epsilon}}' a_\alpha^+ a_\alpha^-
\dgal{v_{i_\alpha} , v_{j_\epsilon}}''  b_\epsilon^+ b_\beta^-  \dgal{v_{j_\beta},c}''  \label{Eq:A2} \\
&+ \sum_{\alpha=1}^r \sum_{\substack{\beta,\epsilon=1 \\ \epsilon<\beta}}^s
\dgal{v_{j_\beta},c}' b_\beta^+ b_\epsilon^-  \dgal{v_{i_\alpha} , v_{j_\epsilon}}' a_\alpha^+ a_\alpha^-
\dgal{v_{i_\alpha} , v_{j_\epsilon}}'' b_{\epsilon+1,\beta-1}^\sim   \dgal{v_{j_\beta},c}''  \label{Eq:A3} \\
&+
\sum_{\alpha=1}^r \sum_{\beta=1}^s \dgal{v_{j_\beta},c}' b_\beta^+ b_\beta^-
\dgal{v_{i_\alpha} , \dgal{v_{j_\beta},c}''}' a_\alpha^+ a_\alpha^- \dgal{v_{i_\alpha} , \dgal{v_{j_\beta},c}''}''  \,.  \label{Eq:A4}
 \end{align}
\end{subequations}
\end{lem}
\begin{proof}
For the reader's convenience, we prove this case.
It suffices to use the Leibniz rules \eqref{Eq:Lei} for $\dgal{-,-}$ before applying the multiplication map.
Thus,
$$\br{b,c}=\sum_\beta \mult(\dgal{v_{j_\beta},c}' b_\beta^+ \otimes b_\beta^- \dgal{v_{j_\beta},c}'')
=\sum_\beta \dgal{v_{j_\beta},c}' b_\beta^+ b_\beta^- \dgal{v_{j_\beta},c}''\,.$$
Hence, we get
\begin{align*}
\br{a,\br{b,c}}=&  \sum_{\beta} \Big[
\br{a,\dgal{v_{j_\beta},c}'} b_\beta^+ b_\beta^- \dgal{v_{j_\beta},c}''
+\dgal{v_{j_\beta},c}' \br{a,b_\beta^+} b_\beta^- \dgal{v_{j_\beta},c}'' \\
&+\dgal{v_{j_\beta},c}' b_\beta^+ \br{a,b_\beta^-} \dgal{v_{j_\beta},c}''
+\dgal{v_{j_\beta},c}' b_\beta^+ b_\beta^- \br{a,\dgal{v_{j_\beta},c}''} \Big]\,,
\end{align*}
and these four terms will give \eqref{Eq:A1}--\eqref{Eq:A4}, respectively. To see this, note for example that in the second term we can use the following expansion:
\begin{align*}
\br{a,b_\beta^+}= \sum_\alpha \dgal{v_{i_\alpha},b_\beta^+}' a_\alpha^+ a_\alpha^- \dgal{v_{i_\alpha},b_\beta^+}''
= \sum_\alpha \sum_{\epsilon>\beta}  b_{\beta+1,\epsilon-1}^\sim \dgal{v_{i_\alpha},v_{j_\epsilon}}' a_\alpha^+ a_\alpha^- \dgal{v_{i_\alpha},v_{j_\epsilon}}'' b_\epsilon^+\,,
\end{align*}
since $b_\beta^+=b_{\beta+1,\epsilon-1}^\sim v_{j_\epsilon} b_\epsilon^+$.
\end{proof}

Exchanging the roles of $a$ and $b$ in the previous lemma, we find:
\begin{subequations} \label{Eq:br-B}
 \begin{align}
\br{b,\br{a,c}}=& \
\sum_{\alpha=1}^r \sum_{\beta=1}^s \dgal{v_{j_\beta} , \dgal{v_{i_\alpha},c}'}'  b_\beta^+ b_\beta^-
\dgal{v_{j_\beta} , \dgal{v_{i_\alpha},c}'}'' a_\alpha^+ a_\alpha^-  \dgal{v_{i_\alpha},c}'' \label{Eq:B1} \\
&+\sum_{\substack{\alpha,\gamma=1 \\ \alpha<\gamma}}^r \sum_{\beta=1}^s
\dgal{v_{i_\alpha},c}' a_{\alpha+1,\gamma-1}^\sim \dgal{v_{j_\beta} , v_{i_\gamma}}' b_\beta^+ b_\beta^-
\dgal{v_{j_\beta} , v_{i_\gamma}}''  a_\gamma^+ a_\alpha^-  \dgal{v_{i_\alpha},c}''  \label{Eq:B2} \\
&+\sum_{\substack{\alpha,\gamma=1 \\ \alpha>\gamma}}^r \sum_{\beta=1}^s
\dgal{v_{i_\alpha},c}'   a_\alpha^+ a_\gamma^- \dgal{v_{j_\beta} , v_{i_\gamma}}' b_\beta^+ b_\beta^-
\dgal{v_{j_\beta} , v_{i_\gamma}}''  a_{\gamma+1,\alpha-1}^\sim  \dgal{v_{i_\alpha},c}''  \label{Eq:B3} \\
&+
\sum_{\alpha=1}^r \sum_{\beta=1}^s \dgal{v_{i_\alpha},c}' a_\alpha^+ a_\alpha^-
\dgal{v_{j_\beta} , \dgal{v_{i_\alpha},c}''}' b_\beta^+ b_\beta^-  \dgal{v_{j_\beta} , \dgal{v_{i_\alpha},c}''}''\,.\label{Eq:B4}
 \end{align}
\end{subequations}

\begin{lem}
The following holds:
\begin{subequations} \label{Eq:br-C}
 \begin{align}
\br{\br{a,b},c}=& \
 \sum_{\alpha=1}^r \sum_{\substack{\beta,\epsilon=1 \\ \epsilon<\beta}}^s
\dgal{v_{j_\epsilon},c}' b_{\epsilon+1,\beta-1}^\sim \dgal{v_{i_\alpha} , v_{j_\beta}}' a_\alpha^+ a_\alpha^-
\dgal{v_{i_\alpha} , v_{j_\beta}}''  b_\beta^+ b_\epsilon^-  \dgal{v_{j_\epsilon},c}''  \label{Eq:C1} \\
&+
\sum_{\alpha=1}^r \sum_{\beta=1}^s \dgal{\dgal{v_{i_\alpha} , v_{j_\beta}}',c}' a_\alpha^+ a_\alpha^-
\dgal{v_{i_\alpha} , v_{j_\beta}}'' b_\beta^+ b_\beta^-  \dgal{\dgal{v_{i_\alpha} , v_{j_\beta}}',c}''\label{Eq:C2} \\
&+\sum_{\substack{\alpha,\gamma=1 \\ \gamma>\alpha}}^s \sum_{\beta=1}^s
\dgal{v_{i_\gamma},c}'   a_\gamma^+ a_\alpha^- \dgal{v_{i_\alpha} , v_{j_\beta}}'' b_\beta^+ b_\beta^-
\dgal{v_{i_\alpha} , v_{j_\beta}}'  a_{\alpha+1,\gamma-1}^\sim  \dgal{v_{i_\gamma},c}''  \label{Eq:C3} \\
&+\sum_{\substack{\alpha,\gamma=1 \\ \gamma<\alpha}}^s \sum_{\beta=1}^s
\dgal{v_{i_\gamma},c}' a_{\gamma+1,\alpha-1}^\sim  \dgal{v_{i_\alpha} , v_{j_\beta}}'' b_\beta^+ b_\beta^-
\dgal{v_{i_\alpha} , v_{j_\beta}}'  a_\alpha^+ a_\gamma^-   \dgal{v_{i_\gamma},c}''  \label{Eq:C4} \\
&+
\sum_{\alpha=1}^r \sum_{\beta=1}^s \dgal{\dgal{v_{i_\alpha} , v_{j_\beta}}'',c}'  b_\beta^+ b_\beta^-
\dgal{v_{i_\alpha} , v_{j_\beta}}' a_\alpha^+ a_\alpha^- \dgal{\dgal{v_{i_\alpha} , v_{j_\beta}}'',c}''\label{Eq:C5} \\
&+
 \sum_{\alpha=1}^r \sum_{\substack{\beta,\epsilon=1 \\ \epsilon>\beta}}^s
\dgal{v_{j_\epsilon},c}' b_\epsilon^+ b_\beta^- \dgal{v_{i_\alpha} , v_{j_\beta}}' a_\alpha^+ a_\alpha^-
\dgal{v_{i_\alpha} , v_{j_\beta}}''  b_{\beta+1,\epsilon-1}^\sim   \dgal{v_{j_\epsilon},c}'' \,. \label{Eq:C6}
 \end{align}
\end{subequations}
\end{lem}
\begin{proof}
Direct computation using $\br{a,b}=\sum_{\alpha,\beta} b_\beta^- \dgal{v_{i_\alpha} , v_{j_\beta}}' a_\alpha^+ a_\alpha^- \dgal{v_{i_\alpha} , v_{j_\beta}}'' b_\beta^+$.
\end{proof}

The next result requires further assumptions.
\begin{lem}
Assume that $(A,\dgal{-,-})$ is a mixed double Poisson algebra of weight $\underline{\lambda}\in \K^d$.
For $1\leq i,j\leq d$ and $c\in A$, we have
\begin{equation} \label{Eq:Prep1}
 \begin{aligned}
&-\dgal{v_j,\dgal{v_i,c}}_L + \dgal{v_i,\dgal{v_j,c}}_R - \dgal{\dgal{v_i,v_j}^\circ,c}_L \\
=& \frac{\lambda_j+\lambda_i}{2}\, v_j \otimes_1 \dgal{v_i,c} + \frac{\lambda_j-\lambda_i}{2}\, 1\otimes_1 (\dgal{v_i,c} \ast v_j)\,.
 \end{aligned}
\end{equation}
\end{lem}
\begin{proof}
We can write the left-hand side of \eqref{Eq:Prep1} as
\begin{align*}
&-\DJac(v_j,v_i,c)
- \frac{\lambda_j+\lambda_i}{2} \dgal{v_j \otimes v_i - v_i \otimes v_j,c}_L
+ \frac{\lambda_j-\lambda_i}{2} \dgal{v_i v_j  \otimes 1 ,c}_L \\
=&-\DJac(v_j,v_i,c)
- \frac{\lambda_j+\lambda_i}{2}\, v_i  \otimes_1 \dgal{v_j,c} + \frac{\lambda_j+\lambda_i}{2}\, v_j \otimes_1 \dgal{ v_i,c} \\
&+ \frac{\lambda_j-\lambda_i}{2} 1\otimes_1 (v_i \ast \dgal{v_j   ,c} )
+ \frac{\lambda_j-\lambda_i}{2} 1\otimes_1 (\dgal{v_i,c} \ast v_j)\,,
\end{align*}
where we used \eqref{Eq:intr1} to obtain the first line, and then the Leibniz rules \eqref{Eq:Lei}.
Due to \eqref{Eq:DJac-c}, the first, second and fourth term cancel out.
\end{proof}

\subsection{Proof of Proposition \ref{Pr:Jac}} \label{ss:PfJac}

We adapt the proof of \cite[Thm.~9]{GG22} to our more general setting.
For $a,b,c\in A$, we want to check that \eqref{Eq:Jac} holds. By linearity, we can assume that $a,b$ are of the form \eqref{Eq:abForm}. Then, the Jacobi identity \eqref{Eq:Jac} amounts to checking
\begin{equation*}
 \eqref{Eq:br-A} - \eqref{Eq:br-B} - \eqref{Eq:br-C} = 0\,.
\end{equation*}
We directly see that $\eqref{Eq:A2}-\eqref{Eq:C1}=0$ and $\eqref{Eq:A3}-\eqref{Eq:C6}=0$. Next, by introducing
\begin{equation*}
 \mult_3:A\otimes A \otimes A \to A, \quad a_1 \otimes a_2 \otimes a_3 \mapsto  a_1  a_2 a_3\,,
\end{equation*}
and recalling \eqref{Eq:dbr3}, we can write
\begin{align*}
 \eqref{Eq:A1}&=\sum_{\alpha=1}^r \sum_{\beta=1}^s \mult_3 \left(
(1\otimes a_\alpha^+ a_\alpha^- \otimes b_\beta^+ b_\beta^-) \ \dgal{v_{i_\alpha} , \dgal{v_{j_\beta},c}}_L
\right)\,, \\
 \eqref{Eq:B4}&=\sum_{\alpha=1}^r \sum_{\beta=1}^s \mult_3 \left(
(1\otimes a_\alpha^+ a_\alpha^- \otimes b_\beta^+ b_\beta^-) \ \dgal{v_{j_\beta} , \dgal{v_{i_\alpha},c}}_R
\right)\,, \\
 \eqref{Eq:C2}&=\sum_{\alpha=1}^r \sum_{\beta=1}^s \mult_3 \left(
(1\otimes a_\alpha^+ a_\alpha^- \otimes b_\beta^+ b_\beta^-) \ \dgal{\dgal{v_{i_\alpha} , v_{j_\beta}},c}_L
\right)\,.
\end{align*}
Thus, \eqref{Eq:DJac} and the condition \eqref{Eq:DJac-c} yield
 \begin{align}
&\eqref{Eq:A1}-\eqref{Eq:B4}-\eqref{Eq:C2}=
\sum_{\alpha=1}^r \sum_{\beta=1}^s \mult_3 \left(
(1\otimes a_\alpha^+ a_\alpha^- \otimes b_\beta^+ b_\beta^-) \ \DJac(v_{i_\alpha} , v_{j_\beta},c) \right) \nonumber \\
=&-\sum_{\alpha=1}^r \sum_{\beta=1}^s \frac{\lambda_{i_\alpha}+\lambda_{j_\beta}}{2} \mult_3 \left(
(1\otimes a_\alpha^+ a_\alpha^- \otimes b_\beta^+ b_\beta^-)
(\dgal{v_{i_\alpha},c}' \otimes v_{j_\beta}\otimes \dgal{v_{i_\alpha},c}'') \right)  \nonumber \\
&+\sum_{\alpha=1}^r \sum_{\beta=1}^s \frac{\lambda_{i_\alpha}-\lambda_{j_\beta}}{2} \ \mult_3 \left(
(1\otimes a_\alpha^+ a_\alpha^- \otimes b_\beta^+ b_\beta^-)
(\dgal{v_{i_\alpha},c}' \otimes 1\otimes v_{j_\beta}\dgal{v_{i_\alpha},c}'') \right) \nonumber \\
=&-\sum_{\alpha=1}^r \sum_{\beta=1}^s \frac{ \lambda_{j_\beta}+\lambda_{j_{\beta+1}} }{2} \
\dgal{v_{i_\alpha},c}' a_\alpha^+ a_\alpha^- b_\beta^+ b_\beta^- v_{j_\beta}  \dgal{v_{i_\alpha},c}''\,. \label{Eq:A1B4C2}
 \end{align}
To obtain the last equality, we used
\begin{equation} \label{Eq:vbb}
 v_{j_\beta} b_\beta^+ b_\beta^- = b_{\beta-1}^+ b_{\beta-1}^- v_{j_{\beta-1}}\,,
\end{equation}
to  sum over $\beta-1$ in the first sum. Here and below, the index $\beta$ is understood modulo $s$, and the identity \eqref{Eq:vbb} holds with $\beta=1$ since $v_{j_1}b_1^+=b=b_{s}^- v_{j_s}$; we do the same with the indices  $\alpha,\gamma$ modulo $r$.
With the same reasoning, we can write
\begin{align*}
 \eqref{Eq:A4}&=\sum_{\alpha=1}^r \sum_{\beta=1}^s \mult_3 \left(
(1\otimes b_\beta^+ b_\beta^- \otimes a_\alpha^+ a_\alpha^- ) \ \dgal{v_{i_\alpha} , \dgal{v_{j_\beta},c}}_R
\right)\,, \\
 \eqref{Eq:B1}&=\sum_{\alpha=1}^r \sum_{\beta=1}^s \mult_3 \left(
(1\otimes b_\beta^+ b_\beta^- \otimes a_\alpha^+ a_\alpha^- ) \ \dgal{v_{j_\beta} , \dgal{v_{i_\alpha},c}}_L
\right)\,, \\
 \eqref{Eq:C5}&=\sum_{\alpha=1}^r \sum_{\beta=1}^s \mult_3 \left(
(1 \otimes b_\beta^+ b_\beta^- \otimes a_\alpha^+ a_\alpha^- ) \ \dgal{\dgal{v_{i_\alpha} , v_{j_\beta}}^\circ,c}_L
\right)\,.
\end{align*}
Therefore, we obtain from \eqref{Eq:Prep1} that
\begin{align}
& \eqref{Eq:A4} - \eqref{Eq:B1} - \eqref{Eq:C5}   \nonumber \\
=&\sum_{\alpha=1}^r \sum_{\beta=1}^s \frac{\lambda_{j_\beta}+\lambda_{i_\alpha}}{2} \mult_3 \left(
(1\otimes b_\beta^+ b_\beta^- \otimes a_\alpha^+ a_\alpha^- ) \ (\dgal{v_{i_\alpha} ,c}' \otimes v_{j_\beta} \otimes \dgal{v_{i_\alpha} ,c}'')
\right) \nonumber \\
&+\sum_{\alpha=1}^r \sum_{\beta=1}^s \frac{\lambda_{j_\beta}-\lambda_{i_\alpha}}{2} \mult_3 \left(
(1\otimes b_\beta^+ b_\beta^- \otimes a_\alpha^+ a_\alpha^- ) \ (\dgal{v_{i_\alpha} ,c}' v_{j_\beta}  \otimes 1 \otimes \dgal{v_{i_\alpha} ,c}'')
\right)\nonumber \\
=& \sum_{\alpha=1}^r \sum_{\beta=1}^s \frac{ \lambda_{j_\beta}+\lambda_{j_{\beta+1}} }{2} \
\dgal{v_{i_\alpha} ,c}' b_\beta^+ b_\beta^- v_{j_\beta} a_\alpha^+ a_\alpha^- \dgal{v_{i_\alpha} ,c}'' \,,  \label{Eq:A4B1C5}
\end{align}
where we used \eqref{Eq:vbb}. Thus, our aim reduces to checking
\begin{equation*}
\eqref{Eq:A1B4C2}+\eqref{Eq:A4B1C5}-\eqref{Eq:B2}-\eqref{Eq:B3}-\eqref{Eq:C3}-\eqref{Eq:C4} = 0\,.
\end{equation*}

\begin{lem} \label{Lem:B2C4}
 The following holds:
\begin{subequations}
 \begin{align}
-\eqref{Eq:B2}-\eqref{Eq:C4}=&
-\sum_{\alpha=1}^r \sum_{\beta=1}^s \frac{ \lambda_{j_\beta}+\lambda_{j_{\beta+1}} }{2} \
\dgal{v_{i_\alpha} ,c}' b_\beta^+ b_\beta^- v_{j_\beta} a_\alpha^+ a_\alpha^- \dgal{v_{i_\alpha} ,c}'' \label{Eq:E1}  \\
&+\sum_{\alpha=1}^r \sum_{\beta=1}^s \frac{ \lambda_{j_\beta}+\lambda_{j_{\beta+1}} }{2} \
\dgal{v_{i_\alpha} ,c}' a_\alpha^+ b_\beta^+ b_\beta^- v_{j_\beta}  a_\alpha^- \dgal{v_{i_\alpha} ,c}''\,.  \label{Eq:E2}
 \end{align}
\end{subequations}
\end{lem}
\begin{proof}
Notice that \eqref{Eq:intr1} entails
\begin{subequations}
 \begin{align}
\eqref{Eq:B2}&+\eqref{Eq:C4}=
\sum_{\substack{\alpha,\gamma=1 \\ \alpha<\gamma}}^r \sum_{\beta=1}^s
\dgal{v_{i_\alpha},c}' a_{\alpha+1,\gamma-1}^\sim
\Big( \dgal{v_{j_\beta} , v_{i_\gamma}}' b_\beta^+ b_\beta^- \dgal{v_{j_\beta} , v_{i_\gamma}}'' \nonumber \\
& + \dgal{v_{i_\gamma},v_{j_\beta}}'' b_\beta^+ b_\beta^- \dgal{v_{i_\gamma},v_{j_\beta}}'\Big)
a_\gamma^+ a_\alpha^-  \dgal{v_{i_\alpha},c}'' \nonumber \\
=& \ \sum_{\substack{\alpha,\gamma=1 \\ \alpha<\gamma}}^r \sum_{\beta=1}^s
\frac{ \lambda_{j_\beta}+\lambda_{i_\gamma} }{2}  \ \dgal{v_{i_\alpha},c}' a_{\alpha+1,\gamma-1}^\sim
v_{j_\beta} b_\beta^+ b_\beta^- v_{i_\gamma} a_\gamma^+ a_\alpha^-  \dgal{v_{i_\alpha},c}'' \label{Eq:B2a} \\
&- \sum_{\substack{\alpha,\gamma=1 \\ \alpha<\gamma}}^r \sum_{\beta=1}^s
\frac{ \lambda_{j_\beta}+\lambda_{i_\gamma} }{2}  \ \dgal{v_{i_\alpha},c}' a_{\alpha+1,\gamma-1}^\sim
v_{i_\gamma} b_\beta^+ b_\beta^- v_{j_\beta} a_\gamma^+ a_\alpha^-  \dgal{v_{i_\alpha},c}'' \label{Eq:B2b} \\
&+ \sum_{\substack{\alpha,\gamma=1 \\ \alpha<\gamma}}^r \sum_{\beta=1}^s
\frac{ \lambda_{j_\beta}-\lambda_{i_\gamma} }{2}  \ \dgal{v_{i_\alpha},c}' a_{\alpha+1,\gamma-1}^\sim
 b_\beta^+ b_\beta^- v_{j_\beta} v_{i_\gamma} a_\gamma^+ a_\alpha^-  \dgal{v_{i_\alpha},c}'' \label{Eq:B2c} \\
&- \sum_{\substack{\alpha,\gamma=1 \\ \alpha<\gamma}}^r \sum_{\beta=1}^s
\frac{ \lambda_{j_\beta}-\lambda_{i_\gamma} }{2}  \ \dgal{v_{i_\alpha},c}' a_{\alpha+1,\gamma-1}^\sim
v_{i_\gamma} v_{j_\beta} b_\beta^+ b_\beta^-  a_\gamma^+ a_\alpha^-  \dgal{v_{i_\alpha},c}'' \label{Eq:B2d}\,.
\end{align}
\end{subequations}
We can write thanks to \eqref{Eq:vbb}
\begin{align*}
& \eqref{Eq:B2a}+\eqref{Eq:B2c}=
\sum_{\substack{\alpha,\gamma=1 \\ \alpha<\gamma}}^r \sum_{\beta=1}^s
\frac{ \lambda_{j_\beta}+\lambda_{j_{\beta+1}} }{2} \ \dgal{v_{i_\alpha},c}' a_{\alpha+1,\gamma-1}^\sim
 b_\beta^+ b_\beta^- v_{j_\beta} v_{i_\gamma} a_\gamma^+ a_\alpha^-  \dgal{v_{i_\alpha},c}'' \\
&= \sum_{\alpha=1}^{r-1} \sum_{\beta=1}^s
\frac{ \lambda_{j_\beta}+\lambda_{j_{\beta+1}} }{2} \ \dgal{v_{i_\alpha},c}'
 b_\beta^+ b_\beta^- v_{j_\beta} a_\alpha^+ a_\alpha^-  \dgal{v_{i_\alpha},c}'' \\
&\quad +\sum_{\alpha=1}^{r-1} \sum_{\gamma=\alpha+1}^{r-1} \sum_{\beta=1}^s
\frac{ \lambda_{j_\beta}+\lambda_{j_{\beta+1}} }{2} \ \dgal{v_{i_\alpha},c}' a_{\alpha+1,\gamma}^\sim
 b_\beta^+ b_\beta^- v_{j_\beta} a_\gamma^+ a_\alpha^-  \dgal{v_{i_\alpha},c}'' \,,
\end{align*}
where we separated the case $\gamma=\alpha+1$ from the rest (which is then summed over $\gamma-1$) to get the second equality.
Similarly,
\begin{align*}
& \eqref{Eq:B2b}+\eqref{Eq:B2d}=
- \sum_{\substack{\alpha,\gamma=1 \\ \alpha<\gamma}}^r \sum_{\beta=1}^s
\frac{ \lambda_{j_\beta}+\lambda_{j_{\beta+1}} }{2} \ \dgal{v_{i_\alpha},c}' a_{\alpha+1,\gamma}^\sim
\, b_\beta^+ b_\beta^- v_{j_\beta} a_\gamma^+ a_\alpha^-  \dgal{v_{i_\alpha},c}'' \\
&= - \sum_{\alpha=1}^{r-1} \sum_{\beta=1}^s
\frac{ \lambda_{j_\beta}+\lambda_{j_{\beta+1}} }{2}\ \dgal{v_{i_\alpha},c}' a_{\alpha}^+
 b_\beta^+ b_\beta^- v_{j_\beta}  a_\alpha^-  \dgal{v_{i_\alpha},c}'' \\
&\quad - \sum_{\alpha=1}^{r-1} \sum_{\gamma=\alpha+1}^{r-1} \sum_{\beta=1}^s
\frac{ \lambda_{j_\beta}+\lambda_{j_{\beta+1}} }{2} \ \dgal{v_{i_\alpha},c}' a_{\alpha+1,\gamma}^\sim
\, b_\beta^+ b_\beta^- v_{j_\beta} a_\gamma^+ a_\alpha^-  \dgal{v_{i_\alpha},c}''\,.
\end{align*}
Gathering these expressions, we get after cancellations

\begin{align*}
&-\eqref{Eq:B2}-\eqref{Eq:C4}=
-\sum_{\alpha=1}^{r-1} \sum_{\beta=1}^s \frac{ \lambda_{j_\beta}+\lambda_{j_{\beta+1}} }{2} \
\dgal{v_{i_\alpha} ,c}' b_\beta^+ b_\beta^- v_{j_\beta} a_\alpha^+ a_\alpha^- \dgal{v_{i_\alpha} ,c}''   \\
& \ +\sum_{\alpha=1}^{r-1} \sum_{\beta=1}^s \frac{ \lambda_{j_\beta}+\lambda_{j_{\beta+1}} }{2} \
\dgal{v_{i_\alpha} ,c}' a_\alpha^+ b_\beta^+ b_\beta^- v_{j_\beta}  a_\alpha^- \dgal{v_{i_\alpha} ,c}'' \,.
\end{align*}
We find precisely \eqref{Eq:E1}--\eqref{Eq:E2} if we can add $\alpha=r$ in the above sums.
This is true since $a_r^+=1$ entails
$b_\beta^+ b_\beta^- v_{j_\beta} a_r^+ a_r^- = a_r^+ b_\beta^+ b_\beta^- v_{j_\beta}  a_r^-$.
\end{proof}

\begin{lem} \label{Lem:B3C3}
 The following holds:
\begin{subequations}
 \begin{align}
-\eqref{Eq:B3}-\eqref{Eq:C3}=&
-\sum_{\alpha=1}^r \sum_{\beta=1}^s \frac{ \lambda_{j_\beta}+\lambda_{j_{\beta+1}} }{2} \
\dgal{v_{i_\alpha} ,c}' a_\alpha^+ b_\beta^+ b_\beta^- v_{j_\beta}  a_\alpha^- \dgal{v_{i_\alpha} ,c}'' \label{Eq:F1}  \\
&+\sum_{\alpha=1}^r \sum_{\beta=1}^s \frac{ \lambda_{j_\beta}+\lambda_{j_{\beta+1}} }{2} \
\dgal{v_{i_\alpha} ,c}' a_\alpha^+  a_\alpha^- b_\beta^+ b_\beta^- v_{j_\beta}  \dgal{v_{i_\alpha} ,c}''\,.  \label{Eq:F2}
 \end{align}
\end{subequations}
\end{lem}
\begin{proof}
Direct computation similar to Lemma \ref{Lem:B2C4}.
\end{proof}

We can now conclude since the remaining terms cancel out as follows:
\begin{equation*}
 \eqref{Eq:E2}+\eqref{Eq:F1}=0, \quad \eqref{Eq:A1B4C2}+\eqref{Eq:F2}=0, \quad \eqref{Eq:A4B1C5}+\eqref{Eq:E1}=0.
\end{equation*}


\section{Localisation} \label{S:Loc}

In this section, $A$ is arbitrary. This means for us that $A$ comes equipped with a projection
\begin{equation*}
 \pi : \K\langle v_1,\ldots,v_d\rangle \longrightarrow A
\end{equation*}
which provides $d$ generators for $A$. Of course, there may be relations given by elements in $\ker \pi$.
Assuming from now on that such a choice of generators is fixed and that we have a linear map
$\dgal{-,-}:A^{\otimes 2} \to A^{\otimes 2}$, we can talk about the pair $(A,\dgal{-,-})$ being a mixed double (Poisson) algebra of some weight $\underline{\lambda}\in \K^d$, by considering Definitions~\ref{Def:weiDA} and~\ref{Def:PwDA} with the ordered generators $v_1,\ldots,v_d$ provided by $\pi$.

Remark that the proofs of Propositions \ref{Pr:wsk} and \ref{Pr:Jac} only require the existence of generators $v_1,\ldots,v_d$ for $A$ such that \eqref{Eq:intr1} and \eqref{Eq:DJac-v} hold. In particular, these statements are valid for an arbitrary $\K$-algebra $A$.
Hence Theorem~\ref{Thm:MAIN} still holds, i.e. a mixed double Poisson algebra $(A,\dgal{-,-})$ is naturally equipped with a modified double bracket.

The next result allows to construct examples on algebras of noncommutative Laurent polynomials by localisation.
It relies on the fact that the Leibniz rules \eqref{Eq:Lei} allow to define $\dgal{-,-}$ on inverses: for $a\in A$ and for $b\in A$ invertible, we must have
\begin{equation} \label{Eq:Preloc}
 \dgal{a,b^{-1}}=-b^{-1} \cdot \dgal{a,b} \cdot b^{-1}\,, \quad
 \dgal{b^{-1},a}=-b^{-1} \ast \dgal{b,a} \ast b^{-1}\,,
\end{equation}
because $\dgal{a,1}=0=\dgal{1,a}$ by $\K$-linearity. (We use the notation \eqref{Eq:Bimod}).

\begin{thm} \label{Thm:loc}
Fix $I=(i_1,\ldots,i_r)$ a sequence of distinct indices in $\{1,\ldots,d\}$, such that $1\leq r \leq d$.
Let $A=\K \langle v_1,\ldots,v_d, v_{i_1}^{-1}, \ldots,v_{i_r}^{-1}\rangle$ be given by the projection
 \begin{equation} \label{Eq:mapPi}
\begin{aligned}
  \pi : \K\langle v_1,\ldots,v_d,v_{d+1},\ldots,v_{d+r}\rangle \to A\,, \quad
  \pi(v_j)=\left\{
\begin{array}{ll}
 v_j & j\leq d, \\
 v_{i_{\alpha}}^{-1} & j=d+\alpha, \  1\leq \alpha \leq r,
\end{array}
  \right.
\end{aligned}
\end{equation}
whose kernel is the two-sided ideal $(v_{i_{\alpha}}v_{d+\alpha}-1,\, v_{d+\alpha}v_{i_{\alpha}}-1 \mid 1\leq \alpha \leq r)$.

If $\dgal{-,-}$ is a mixed double (Poisson) algebra structure of weight $\underline{\lambda}=(\lambda_{1},\ldots,\lambda_{d})\in \K^d$ on $\K\langle v_1,\ldots,v_d\rangle$, then it uniquely extends to $A$ into such a structure of weight
\begin{align*}
 (\lambda_{1},\ldots,\lambda_{d},-\lambda_{i_1},\ldots,-\lambda_{i_r})\, \in \K^{d+r}\,.
\end{align*}
\end{thm}
\begin{proof}
 We use the presentation given by \eqref{Eq:mapPi} for the proof, so that
$$\lambda_{d+\alpha}:=-\lambda_{i_\alpha}, \qquad 1\leq \alpha \leq r\,.$$
The operation $\dgal{-,-}$ uniquely extends to $A$ by the Leibniz rules \eqref{Eq:Lei} as (cf. \eqref{Eq:Preloc})
\begin{equation}
 \dgal{a,v_{d+\alpha}}=-v_{d+\alpha} \cdot \dgal{a,v_{i_\alpha}} \cdot v_{d+\alpha}\,, \quad
 \dgal{v_{d+\alpha},a}=-v_{d+\alpha} \ast \dgal{v_{i_\alpha},a} \ast v_{d+\alpha}\,. \label{Eq:Loc3}
\end{equation}

First, we need to check that \eqref{Eq:intr1} holds; this is clear if $1\leq i,j\leq d$ by assumption.
If $i=d+\alpha$ and $1\leq j \leq d$, one has
\begin{align*}
&\dgal{v_{d+\alpha},v_j} + \dgal{v_j , v_{d+\alpha}}^\circ
= - v_{d+\alpha} \ast(\dgal{v_{i_\alpha},v_j} + \dgal{v_j , v_{i_\alpha}}^\circ) \ast v_{d+\alpha} \\
 =&- \frac{\lambda_{i_\alpha}+\lambda_{j}}{2} (1 \otimes v_{d+\alpha} v_j - v_jv_{d+\alpha} \otimes 1)
- \frac{\lambda_{i_\alpha}-\lambda_{j}}{2} \, (v_{d+\alpha}\otimes v_j - v_j \otimes v_{d+\alpha})  \\
=&
 \ \frac{\lambda_{d+\alpha}+\lambda_{j}}{2}\, (v_{d+\alpha}\otimes v_j - v_j \otimes v_{d+\alpha})
+ \frac{\lambda_{d+\alpha}-\lambda_{j}}{2} (1 \otimes v_{d+\alpha} v_j - v_jv_{d+\alpha} \otimes 1)
\end{align*}
where we used \eqref{Eq:Loc3} and that \eqref{Eq:intr1} holds for the pair $(v_{i_\alpha},v_j)$.
If $1\leq i \leq d$ and $j=d+\alpha$, this holds by applying $(-)^\circ$ to the previous case.
For the last case, we obtain using \eqref{Eq:Loc3}, and \eqref{Eq:intr1} in the same way
\begin{align*}
&\dgal{v_{d+\alpha},v_{d+\beta}} + \dgal{v_{d+\beta} , v_{d+\alpha}}^\circ
= v_{d+\alpha} \ast(v_{d+\beta} \cdot (\dgal{v_{i_\alpha},v_{i_\beta}} + \dgal{v_{i_\beta} , v_{i_\alpha}}^\circ) \cdot v_{d+\beta} )\ast v_{d+\alpha} \\
 =& \frac{\lambda_{i_\alpha}+\lambda_{i_\beta}}{2} (v_{d+\beta} \otimes v_{d+\alpha} - v_{d+\alpha} \otimes v_{d+\beta})
+ \frac{\lambda_{i_\alpha}-\lambda_{i_\beta}}{2}  \, (v_{d+\beta} v_{d+\alpha}\otimes 1 - 1 \otimes v_{d+\alpha} v_{d+\beta})  \\
=&
\frac{\lambda_{d+\alpha}+\lambda_{d+\beta}}{2} \, (v_{d+\alpha}\otimes v_{d+\beta} - v_{d+\beta} \otimes v_{d+\alpha})
+ \frac{\lambda_{d+\alpha}-\lambda_{d+\beta}}{2}  (1 \otimes v_{d+\alpha} v_{d+\beta} - v_{d+\beta}v_{d+\alpha} \otimes 1)\,,
\end{align*}
which completes the verification.

In the Poisson case, we need to check \eqref{Eq:DJac-v}. This is clear if $1\leq i,j,k\leq d$ by assumption.
Let us then consider $1\leq i,j\leq d$ and $k=d+\alpha$.
By Lemma \ref{Lem:DJac} and $\K$-linearity,
\begin{align*}
 \DJac(v_i,v_j,v_{d+\alpha})=-(v_{d+\alpha}\otimes 1 \otimes 1) \, \DJac(v_i,v_j,v_{i_\alpha}) \, (1\otimes 1 \otimes v_{d+\alpha})\,.
\end{align*}
Since \eqref{Eq:DJac-v} holds in the case $1\leq i,j,i_\alpha\leq d$, we find by \eqref{Eq:Loc3}
\begin{align*}
 \DJac(v_i,v_j,v_{d+\alpha})=&
\frac{\lambda_{i}+\lambda_{j}}{2} \, v_j\otimes_1 (v_{d+\alpha}\cdot \dgal{v_i,v_{i_\alpha}} \cdot  v_{d+\alpha}) \\
&- \frac{\lambda_{i}-\lambda_{j}}{2}  \, 1\otimes_1 (v_j \ast v_{d+\alpha} \cdot \dgal{v_i,v_{i_\alpha}} \cdot  v_{d+\alpha} )  \\
=& -\frac{\lambda_{i}+\lambda_{j}}{2} \, v_j\otimes_1 \dgal{v_i,v_{d+\alpha}}
+\frac{\lambda_{i}-\lambda_{j}}{2}  \, 1\otimes_1 (v_j \ast \dgal{v_i,v_{d+\alpha}} )\,.
\end{align*}
Therefore \eqref{Eq:DJac-v} holds for $1\leq i,j\leq d$ and any $1\leq k \leq d+r$.
In turn, we get that \eqref{Eq:DJac-v} holds for $1\leq i\leq d$ and any $1\leq j,k \leq d+r$ by analogous computations using the second derivation rule mentioned in Lemma \ref{Lem:DJac}.

It remains to check \eqref{Eq:DJac-v} when $i=d+\alpha$.
By \eqref{Eq:DJac-DerNot}, note that we can write
\begin{equation}
 \begin{aligned} \label{Eq:Loc4}
&\DJac(v_{d+\alpha},v_j,v_k)= -(1\otimes v_{d+\alpha}\otimes 1) \DJac(v_{i_\alpha},v_j,v_k) (v_{d+\alpha}\otimes 1 \otimes 1) \\
 &\qquad +\big( 1\otimes (\dgal{v_j,v_{d+\alpha}} + \dgal{v_{d+\alpha},v_j}^\circ) \big)
 (\dgal{v_{i_\alpha},v_k}'v_{d+\alpha}  \otimes 1 \otimes \dgal{v_{i_\alpha},v_k}'') \,.
 \end{aligned}
\end{equation}
Write the two terms appearing on the right-hand side of \eqref{Eq:Loc4} as $\mathtt{T}_1$ and $\mathtt{T}_2$.
Using \eqref{Eq:DJac-v} (with the first index in $\{1,\ldots,d\}$), we rewrite $\mathtt{T}_1$ as
\begin{align*}
&\mathtt{T}_1=
\frac{\lambda_{i_\alpha}+\lambda_{j}}{2} \, (v_{d+\alpha}v_j)\otimes_1 (\dgal{v_{i_\alpha},v_k}\ast v_{d+\alpha})
- \frac{\lambda_{i_\alpha}-\lambda_{j}}{2}  \, v_{d+\alpha}\otimes_1 (v_j \ast \dgal{v_{i_\alpha},v_k} \ast v_{d+\alpha})\\
&=
-\frac{\lambda_{d+\alpha}-\lambda_{j}}{2}\, (v_{d+\alpha}v_j)\otimes_1 (\dgal{v_{i_\alpha},v_k}\ast v_{d+\alpha})
+\frac{\lambda_{d+\alpha}+\lambda_{j}}{2} \, v_{d+\alpha}\otimes_1 (v_j \ast \dgal{v_{i_\alpha},v_k} \ast v_{d+\alpha}).
\end{align*}
As we already noticed that \eqref{Eq:intr1} holds, we also get
\begin{align*}
 \mathtt{T}_2&=
\frac{\lambda_j+\lambda_{d+\alpha}}{2}\, v_j\otimes_1 (v_{d+\alpha} \ast \dgal{v_{i_\alpha},v_k} \ast v_{d+\alpha})
- \frac{\lambda_j+\lambda_{d+\alpha}}{2}\, v_{d+\alpha}\otimes_1 (v_j \ast \dgal{v_{i_\alpha},v_k} \ast v_{d+\alpha}) \\
+&
\frac{\lambda_j-\lambda_{d+\alpha}}{2}\, 1\otimes_1 (v_jv_{d+\alpha} \ast \dgal{v_{i_\alpha},v_k} \ast v_{d+\alpha})
- \frac{\lambda_j-\lambda_{d+\alpha}}{2}\, (v_{d+\alpha}v_j)\otimes_1 (\dgal{v_{i_\alpha},v_k}\ast v_{d+\alpha}).
\end{align*}
Summing $\mathtt{T}_1$ and $\mathtt{T}_2$, we obtain
\begin{align*}
\DJac(v_{d+\alpha},v_j,v_k)
=&\frac{\lambda_j+\lambda_{d+\alpha}}{2}\, v_j\otimes_1 (v_{d+\alpha} \ast \dgal{v_{i_\alpha},v_k} \ast v_{d+\alpha}) \\
&+ \frac{\lambda_j-\lambda_{d+\alpha}}{2}\, 1\otimes_1 (v_jv_{d+\alpha} \ast \dgal{v_{i_\alpha},v_k} \ast v_{d+\alpha}) \\
=&-\frac{\lambda_{d+\alpha}+\lambda_j}{2}\, v_j\otimes_1 \dgal{v_{d+\alpha},v_k}
+ \frac{\lambda_{d+\alpha}-\lambda_j}{2}\, 1\otimes_1 (v_j \ast \dgal{v_{d+\alpha},v_k}) ,
\end{align*}
where the second equality holds by \eqref{Eq:Loc3}.
\end{proof}

\begin{cor} \label{Cor:loc}
 If $\dgal{-,-}$ defines a mixed double Poisson algebra structure (of some weight) on $\K \langle v_1,\ldots,v_d\rangle$, then it uniquely extends to such a structure on the algebra of noncommutative Laurent polynomials
 $A=\K \langle v_1^{\pm1},\ldots,v_d^{\pm 1}\rangle$.
In particular, this defines a modified double Poisson bracket on $A$.
\end{cor}

\begin{exmp}
 Take $\lambda=1$ in Proposition \ref{Pr:CL1}. The case $\rho=-\lambda$ with $\alpha=0$, $\beta=\lambda$ reads
 \begin{equation} \label{Eq:Free2-Art}
\dgal{v,v}=0, \quad \dgal{w,w}=0, \quad
\dgal{v,w}=- \, wv \otimes 1 ,  \quad
\dgal{w,v}=\, vw \otimes 1,
 \end{equation}
which defines a mixed double Poisson algebra structure of weight $(1,1)$ on $\K\langle v,w \rangle$.
Hence \eqref{Eq:Free2-Art} also defines a mixed double Poisson algebra structure on
$\K\langle v^{\pm 1},w^{\pm 1} \rangle$ by localisation, and therefore a modified double Poisson bracket, cf. Corollary~\ref{Cor:loc}.
The weight is $(1,1,-1,-1)$ with respect to the generators $v,w,v^{-1},w^{-1}$.
This is Arthamonov's first example \cite[\S3.4]{Art15} (with $v,w$ respectively standing for $u,v$) of modified double Poisson bracket, constructed in relation to the Kontsevich system \cite{EW}.
\end{exmp}


\section{Examples and Arthamonov's conjecture} \label{S:ClassConj}

Recall the notion of a mixed double Poisson algebra of weight $\underline{\lambda}\in \K^d$ as in Definition \ref{Def:PwDA} with $A=\K\langle v_1,\ldots,v_d \rangle$.
We will focus on quadratic mixed double Poisson algebras, i.e. for any $1\leq i,j\leq d$,  $\dgal{v_i,v_j}\in A\otimes A$ has degree $+2$, where the degree is such that $|fg|=|f|+|g|$ for $f,g\in A\otimes A$ homogeneous and $|v_k\otimes 1|=|1\otimes v_k|=1$ for all $k$.

For the trivial weight $(0,\ldots,0)$, we get back Van den Bergh's cyclic skew-symmetry \eqref{Eq:cskew} and the vanishing of $\DJac$ \eqref{Eq:DJac}; this means that we are in the case of a quadratic double Poisson bracket \cite{VdB1} on a free algebra.
These were classified by Odesskii, Rubtsov and Sokolov \cite{ORS}.

Next, consider a homogeneous weight $(\lambda,\ldots,\lambda)$, $\lambda\in \K^\times$.
If $\dgal{-,-}$ restricts to a map $V\otimes V \to V\otimes V$, $V=\oplus_{k=1}^d \K v_k$, we explained in Example \ref{Exmp:GG} that this is an extension of a  $\lambda$-double Lie algebra. Explicit examples can be found in \cite{GG22}.
To get new interesting cases, we need to assume that some of the weights $(\lambda_{j})$ are distinct.
We shall investigate their classification in the Appendix.
Below, we report on consequences of this classification, and we deduce that Arthamonov's conjecture is true.

\subsection{Families of mixed double Poisson algebras}

\begin{exmp}  \label{Exmp:CL1}
The algebra $\K\langle v,w\rangle$ is a mixed double Poisson algebra of weight $(1,-1)$ if it is equipped with the operation $\dgal{-,-}$ satisfying one of the following four conditions
\begin{subequations}
 \begin{align}
&\dgal{v,w}=0 ,  \qquad
&&\dgal{w,v}= - \, 1 \otimes wv + \, vw \otimes 1 ; \\ 
&\dgal{v,w}= \, 1 \otimes vw ,  \qquad
&&\dgal{w,v}= - \, 1 \otimes wv ; \\ 
&\dgal{v,w}=- \, wv \otimes 1 ,  \qquad
&&\dgal{w,v}= \, vw \otimes 1 ; \\ 
&\dgal{v,w}=  \, 1 \otimes vw -  \, wv \otimes 1 ,  \qquad
&&\dgal{w,v}= 0.  
 \end{align}
\end{subequations}
 This follows from the first part of Proposition~\ref{Pr:CL1}.
\end{exmp}

Write $A_d=\K\langle v_1,\ldots,v_d \rangle$, $d\geq 3$.

\begin{exmp}  \label{Exmp:CL3b}
We can view $A_3$ as a mixed double algebra of weight $(1,1,-1)$ by considering the operation $\dgal{-,-}$ uniquely defined by
\begin{equation} \label{Eq:CL3b}
 \begin{aligned}
  &\dgal{v_1,v_1}=0, \quad \dgal{v_2,v_2}=0, \quad \dgal{v_3,v_3}=0, \\
&\dgal{v_1,v_2}= \tilde \alpha_3 \, v_1 \otimes v_2 - \tilde \beta_3 \, v_2 \otimes v_1,  \\
&\dgal{v_2,v_1}=(-1+\tilde \beta_3) \, v_1 \otimes v_2 +  (1-\tilde \alpha_3)\, v_2 \otimes v_1\,, \\
&\dgal{v_1,v_3}= \alpha_2 \, 1 \otimes v_1 v_3 -  \beta_2 \, v_3 v_1 \otimes 1,  \\
&\dgal{v_3,v_1}=(-1+\beta_2) \, 1 \otimes v_3 v_1 +  (1- \alpha_2)\, v_1 v_3 \otimes 1\,, \\
&\dgal{v_2,v_3}= \alpha_1 \, 1 \otimes v_2 v_3 - \beta_1 \, v_3 v_2 \otimes 1,  \\
&\dgal{v_3,v_2}=(-1+\beta_1) \, 1 \otimes v_3 v_2 +  (1-\alpha_1)\, v_2 v_3 \otimes 1\,.
 \end{aligned}
\end{equation}
for $\alpha_1, \alpha_2,\tilde\alpha_3,\beta_1, \beta_2,\tilde\beta_3 \in \K$.
Furthermore, $\dgal{-,-}$ is Poisson when the triples
$(\alpha_1,\alpha_2,\tilde{\beta_3})$ and $(\beta_1,\beta_2,\tilde{\alpha_3})$ take one of the following $6$ values
\begin{equation} \label{Eq:TripCL3a}
 (0,0,0), \ (1,0,0), \  (0,0,1), \ (1,1,0), \  (0,1,1), \ (1,1,1).
\end{equation}
 These cases follow from the classification of Proposition~\ref{Pr:CL3b}.
\end{exmp}

Next, assume that $d>3$.
We introduce for $0\leq \delta \leq d$,
\begin{equation} \label{Eq:1delD}
 \mathbf{1}_{\delta,d}:=(\underbrace{1,\ldots,1}_{\delta},\underbrace{-1,\ldots,-1}_{d-\delta}) \,.
\end{equation}
The next two results will be proved in \ref{ss:PfCLd} and \ref{ss:PfCLd-2}.

\begin{prop}  \label{Pr:CLd}
Fix $d\geq 4$ and $0\leq \delta \leq d$.
The following defines a mixed double Poisson algebra structure on $A_d$ of weight
$\mathbf{1}_{\delta,d}$:
\begin{equation}
  \begin{aligned} \label{Eq:CL4-1}
&\dgal{v_i,v_i}=0, \quad \text{for } 1 \leq i \leq d, \\
&\dgal{v_i,v_j}= \, v_i \otimes v_j -  v_j \otimes v_i,  \quad
\dgal{v_j,v_i}=0\,,  \quad \text{for } 1 \leq i<j \leq \delta\,, \\
&\dgal{v_i,v_k}= \, 1 \otimes v_i v_k - v_k v_i \otimes 1,  \quad
\dgal{v_k,v_i}=0\,, \quad \text{for } 1 \leq i\leq \delta < k \leq d\,, \\
&\dgal{v_k,v_l}= -v_k \otimes v_l +  v_l \otimes v_k,  \quad
\dgal{v_l,v_k}=0\,, \quad \text{for } \delta < k<l \leq d\,.
 \end{aligned}
\end{equation}
In particular, \eqref{Eq:CL4-1} defines a modified double Poisson bracket on $A_d$.
\end{prop}

\begin{prop}  \label{Pr:CLd-2}
Fix $d\geq 4$ and $0\leq \delta \leq d$.
The following defines a mixed double Poisson algebra structure on $A_d$ of weight
$\mathbf{1}_{\delta,d}$:
\begin{equation}
  \begin{aligned} \label{Eq:CL4-2}
&\dgal{v_i,v_i}=0, \quad \text{for } 1 \leq i \leq d, \\
&\dgal{v_i,v_j}= \, v_i \otimes v_j,  \quad
\dgal{v_j,v_i}=-v_i \otimes v_j\,,  \quad \text{for } 1 \leq i<j \leq \delta\,, \\
&\dgal{v_i,v_k}=  - v_k v_i \otimes 1,  \quad
\dgal{v_k,v_i}= \, v_i v_k \otimes 1\,, \quad \text{for } 1 \leq i\leq \delta < k \leq d\,, \\
&\dgal{v_k,v_l}= -v_k \otimes v_l,  \quad
\dgal{v_l,v_k}=\, v_k \otimes v_l\,, \quad \text{for } \delta < k<l \leq d\,.
 \end{aligned}
\end{equation}
In particular, \eqref{Eq:CL4-2} defines a modified double Poisson bracket on $A_d$.
\end{prop}

\begin{rem}
 The case $\delta=d$ in Propositions \ref{Pr:CLd} and \ref{Pr:CLd-2} can be found, respectively, as Examples 1 and 2 in \cite[\S4]{GG22}.
\end{rem}

\subsection{The second instance of Arthamonov's conjecture}

If we look at the operation $\dgal{-,-}^{I}$ defined in \eqref{Eq:MDBI}, we see that
 \begin{align*}
&\dgal{x_1,x_2}^{I}+(\dgal{x_2,x_1}^{I})^\circ= 1\otimes x_1x_2-x_2x_1 \otimes 1, \\
&\dgal{x_2,x_3}^{I}+(\dgal{x_3,x_2}^{I})^\circ=-(x_2 \otimes x_3 - x_3\otimes x_2), \\
&\dgal{x_3,x_1}^{I}+(\dgal{x_1,x_3}^{I})^\circ= -(1\otimes x_3x_1-x_1x_3 \otimes 1).
 \end{align*}
 This defines a mixed double algebra of type
\begin{equation}
 \Lambda^{I}=
\left( \begin{array}{ccc}
    \lambda_{1}&0&0 \\ 0&\lambda_{2}&-1 \\ 0&-1&\lambda_{3}
       \end{array}
\right) \,, \qquad
M^{I}=
\left( \begin{array}{ccc}
    0&1&1 \\ -1&0&0 \\ -1&0&0
       \end{array}
\right)\,.
\end{equation}
Furthermore, it is easily seen to be of weight $(1,-1,-1)$ as  \eqref{Eq:MuLam} holds.

\begin{thm} \label{Thm:AI}
The operation $\dgal{-,-}^{I}$ \eqref{Eq:MDBI} on  $\K\langle x_1,x_2,x_3\rangle$ is a modified double Poisson bracket.
\end{thm}
\begin{proof}
 Up to multiplying $\dgal{-,-}^{I}$ by $-1$ and setting $v_1:=x_3$, $v_2:=x_2$, $v_3:=x_1$, we get a mixed double algebra structure of weight $(1,1,-1)$ that reads:
\begin{equation} \label{Eq:MDBIbis}
 \begin{aligned}
&\dgal{v_1,v_2}=-v_2 \otimes v_1,  \quad &&\dgal{v_2,v_1}=v_2 \otimes v_1, \\
&\dgal{v_1,v_3}=1\otimes v_1v_3,  \quad &&\dgal{v_3,v_1}=-1 \otimes v_3v_1, \\
&\dgal{v_2,v_3}=-v_3 v_2 \otimes 1, \quad &&\dgal{v_3,v_2} = v_2 v_3 \otimes 1,
 \end{aligned}
\end{equation}
where zero brackets are omitted. If we take in \eqref{Eq:CL3b}  the constants
\begin{equation*}
 (\alpha_1,\alpha_2,\tilde\beta_3)= (0,1,1), \quad
 (\beta_1,\beta_2,\tilde\alpha_3) = (1,0,0),
\end{equation*}
we reproduce \eqref{Eq:MDBIbis}. By Proposition~\ref{Pr:CL3b}, this is a mixed double Poisson algebra; hence $\dgal{-,-}^{I}$ is a modified double Poisson bracket by Theorem \ref{Thm:MAIN}.
\end{proof}

\subsection{Open problems}

Based on the constructions carried out in the previous subsections, let us list some questions that require further investigation.

\begin{prob}
Put $A_d=\K\langle v_1,\ldots,v_d\rangle$.
\begin{enumerate}[label=(\alph*)]
 \item
 Does there exist a mixed double Poisson algebra structure on $A_2$ of weight $(\lambda,\rho)$ with $\lambda \neq \pm \rho$?
 \item
Does there exist a mixed double Poisson algebra structure on $A_d$ of some weight $\underline{\lambda}$ where the self-brackets of generators can be nonzero? (I.e. find examples where $\dgal{v_i,v_i}\neq 0$ for $v_i$ a generator of $A_d$ in the considered presentation).
\item
For $d\geq 2$ and any $\underline{\lambda}\in \K^d$,
does there exist a mixed double Poisson algebra structure on $A_d$ of that weight?
\end{enumerate}
\end{prob}

\begin{prob}
 Find examples of mixed double Poisson algebras that are not simply obtained by localisation/quotient of such a structure on a free algebra.
\end{prob}

\begin{prob}
Can one define analogous structures yielding modified double Poisson brackets where the failure to satisfy the cyclic skew-symmetry, cf.~\eqref{Eq:mixDA},
is homogeneous but not quadratic?
\end{prob}

\begin{prob}
Reformulate the conditions \eqref{Eq:mixDA} and \eqref{Eq:DJac-v} in terms of an operator $R$ (cf. \eqref{Eq:RBop})
in such a way that, for the case $(\lambda,\ldots,\lambda)$, $R$ is a $\lambda$-skew-symmetric Rota-Baxter operator of weight $\lambda$.
\end{prob}


\appendix

\section{Some classification results} \label{App:Class}

\subsection{Case \texorpdfstring{$d=2$}{d=2}}

Fix $\lambda,\rho\in \K$ and $A_2=\K\langle v,w\rangle$. A mixed double algebra of weight $(\lambda,\rho)$
must satisfy $\dgal{v,v} =- \dgal{v,v}^\circ$, $\dgal{w,w} =- \dgal{w,w}^\circ$, and
\begin{equation} \label{Eq:mixDA-2}
 \dgal{v,w} + \dgal{w,v}^\circ
 = \frac{\lambda+\rho}{2} (v \otimes w - w \otimes v)
+ \frac{\lambda-\rho}{2} \, (1\otimes vw - wv \otimes 1)\,.
\end{equation}
Furthermore, it is Poisson when the following Poisson conditions hold:
\begin{subequations}
 \begin{align}
 &\DJac(v,v,v)=-\lambda \ v\otimes_1 \dgal{v,v}\,, \qquad
\DJac(w,w,w)=-\rho \ w\otimes_1 \dgal{w,w}\,, \label{Eq:DJ2a}  \\
 &\DJac(v,v,w)=-\lambda \ v\otimes_1 \dgal{v,w}\,, \qquad
 \DJac(w,w,v)=-\rho \ w\otimes_1 \dgal{w,v}\,, \label{Eq:DJ2b} \\
 &\DJac(v,w,v)=-\frac{\lambda+\rho}{2} \ w\otimes_1 \dgal{v,v}
 + \frac{\lambda-\rho}{2} \ 1\otimes_1(w \ast \dgal{v,v})\,, \label{Eq:DJ2c} \\
 &\DJac(v,w,w)=-\frac{\lambda+\rho}{2} \ w\otimes_1 \dgal{v,w}
 + \frac{\lambda-\rho}{2} \ 1\otimes_1(w \ast \dgal{v,w})\,, \label{Eq:DJ2d} \\
 &\DJac(w,v,w)=-\frac{\lambda+\rho}{2} \ v\otimes_1 \dgal{w,w}
- \frac{\lambda-\rho}{2} \ 1\otimes_1(v \ast \dgal{w,w})\,, \label{Eq:DJ2e} \\
 &\DJac(w,v,v)=-\frac{\lambda+\rho}{2} \ v\otimes_1 \dgal{w,v}
- \frac{\lambda-\rho}{2} \ 1\otimes_1(v \ast \dgal{w,v})\,. \label{Eq:DJ2f}
\end{align}
\end{subequations}

To get interesting new examples, we assume that $(\lambda,\rho)\neq(0,0)$.

A first classification on $A_2$ is possible when
\begin{equation} \label{Eq:CL1vv}
 \dgal{v,v}=0, \qquad \dgal{w,w}=0,
\end{equation}
so that \eqref{Eq:DJ2a} holds directly.
We shall make the additional assumption that the mixed term $\dgal{v,w}$ is a linear combination of the $4$ quadratic expressions appearing in \eqref{Eq:mixDA-2}. This yields by \eqref{Eq:mixDA-2},
\begin{subequations} \label{Eq:CL1mix}
 \begin{align}
\dgal{v,w}&=-\frac{\gamma_1}{2} \, v \otimes w + \frac{\gamma_2}{2} \, w \otimes v
- \frac{\gamma_3}{2} \, 1 \otimes vw + \frac{\gamma_4}{2} \, wv \otimes 1\,, \label{Eq:CL1vw} \\
\dgal{w,v}&=-\frac{\lambda + \rho +\gamma_2}{2} \, v \otimes w + \frac{\lambda + \rho +\gamma_1}{2}\, w \otimes v
\nonumber \\
& \quad - \frac{\lambda - \rho +\gamma_4}{2} \, 1 \otimes wv + \frac{\lambda - \rho +\gamma_3}{2} \, vw \otimes 1\,, \label{Eq:CL1wv}
\end{align}
\end{subequations}
for some $\Gamma=(\gamma_1,\ldots,\gamma_4) \in \K^4$.

\begin{lem} \label{Lem:CL1}
 Assume that $\dgal{-,-}$ is nonzero and satisfies \eqref{Eq:CL1vv} and \eqref{Eq:CL1mix}.
Then the Poisson condition \eqref{Eq:DJ2b} holds if and only if we are in one of the following cases:
\begin{enumerate}
 \item $\rho=-\lambda$ and $\Gamma=(0,0,-2\lambda,-2\lambda)$, or $(0,0,-2\lambda,0)$, or $(0,0,0,-2\lambda)$, or $(0,0,0,0)$;
 \item $\rho=+\lambda$ and $\Gamma=(-2\lambda,-2\lambda,0,0)$, or $(-2\lambda,0,0,0)$, or $(0,-2\lambda,0,0)$, or $(0,0,0,0)$;
 \item $\rho=\pm\lambda$ and $\Gamma=(0,-2\lambda,-2\lambda,0)$ or $(-2\lambda,0,0,-2\lambda)$.
\end{enumerate}
\end{lem}
\begin{proof}
Let us examine the first equality in \eqref{Eq:DJ2b}. To write the left-hand side, we compute thanks to \eqref{Eq:dbr3}, \eqref{Eq:CL1vv} and \eqref{Eq:CL1vw}:
\begin{align*}
 \dgal{v,\dgal{v,w}}_L=& -\frac{\gamma_1\gamma_2}{4} v\otimes w\otimes v
 + \frac{\gamma_2^2}{4} w\otimes v\otimes v
 -\frac{\gamma_3\gamma_2}{4} 1\otimes vw\otimes v
 + \frac{\gamma_4\gamma_2}{4} wv\otimes 1\otimes v \\
 -& \ \frac{\gamma_1\gamma_4}{4} v\otimes wv\otimes 1
 + \frac{\gamma_2\gamma_4}{4} w\otimes v^2\otimes 1
 - \frac{\gamma_3\gamma_4}{4} 1\otimes vwv\otimes 1
 + \frac{\gamma_4^2}{4} wv\otimes v\otimes 1 \,, \\
\dgal{v,\dgal{v,w}}_R=& \frac{\gamma_1^2}{4} v\otimes v\otimes w
 - \frac{\gamma_2\gamma_1}{4} v\otimes w\otimes v
 +\frac{\gamma_3\gamma_1}{4} v\otimes 1\otimes vw
 - \frac{\gamma_4\gamma_1}{4} v\otimes wv\otimes 1 \\
 +& \ \frac{\gamma_1\gamma_3}{4} 1\otimes v^2\otimes w
 - \frac{\gamma_2\gamma_3}{4} 1\otimes vw\otimes v
 + \frac{\gamma_3^2}{4} 1\otimes v \otimes vw
 - \frac{\gamma_4 \gamma_3}{4} 1\otimes vwv\otimes 1 \,,
\end{align*}
and $\dgal{\dgal{v,v},w}_L=0$. So \eqref{Eq:DJac} gives for the right-hand side
\begin{equation*}
 \begin{aligned}
\DJac(v,v,w)=& + \frac{\gamma_2^2}{4} w\otimes v\otimes v  + \frac{\gamma_4\gamma_2}{4} (wv\otimes 1\otimes v+ w\otimes v^2\otimes 1)
+ \frac{\gamma_4^2}{4} wv\otimes v\otimes 1 \\
&-\frac{\gamma_1^2}{4} v\otimes v\otimes w -\frac{\gamma_3\gamma_1}{4} (v\otimes 1\otimes vw+ 1\otimes v^2\otimes w)
- \frac{\gamma_3^2}{4} 1\otimes v \otimes vw \,.
 \end{aligned}
\end{equation*}
By \eqref{Eq:CL1vw}, we get for the right-hand side
\begin{equation*}
 -\lambda \ v\otimes_1 \dgal{v,w}=\frac{\lambda\gamma_1}{2} v\otimes v\otimes w
 - \frac{\lambda\gamma_2}{2} w\otimes v\otimes v + \frac{\lambda\gamma_3}{2} 1\otimes v\otimes vw
 - \frac{\lambda\gamma_4}{2} wv\otimes v\otimes 1\,.
\end{equation*}
Matching coefficients, the first equality in \eqref{Eq:DJ2b} holds if and only if
\begin{equation*}
 \gamma_1\gamma_3=0, \quad \gamma_2\gamma_4=0, \quad
 \gamma_i \left(\lambda+\frac{\gamma_i}{2}\right)=0, \ 1\leq i \leq 4\,,
\end{equation*}
which holds if and only if $\Gamma$ is one of the following quadruples:
\begin{equation} \label{Eq:CL1-cond1}
 \begin{aligned}
(0,0,0,0), \quad  (-2\lambda,0,0,0), \quad (0,-2\lambda,0,0), \quad (0,0,-2\lambda,0), \quad (0,0,0,-2\lambda), \\
(-2\lambda,-2\lambda,0,0), \quad (-2\lambda,0,0,-2\lambda), \quad (0,-2\lambda,-2\lambda,0), \quad (0,0,-2\lambda,-2\lambda).
 \end{aligned}
\end{equation}
(So far, there is no condition on $\lambda,\rho$).
Analogous computations entail that the second equality in \eqref{Eq:DJ2b} holds if and only if
\begin{equation} \label{Eq:CL1-cond2}
\begin{aligned}
  (\lambda+\rho+\gamma_2)(\lambda-\rho+\gamma_3)=0,& \qquad (\lambda+\rho+\gamma_1)(\lambda-\rho+\gamma_4)=0, \\
 (\lambda+\rho+\gamma_i)(\lambda-\rho+\gamma_i)=0,& \quad 1\leq i \leq 4\,.
\end{aligned}
\end{equation}
As $\gamma_i\in\{0,-2\lambda\}$, the last condition yields $\lambda^2-\rho^2=0$.
To conclude, it remains to check when $\rho=\lambda$ or $\rho=-\lambda$ which of the cases from \eqref{Eq:CL1-cond1} satisfy the first two equalities in \eqref{Eq:CL1-cond2}; we end up with the different cases \emph{(1)-(3)} from the statement.
\end{proof}

\begin{prop}  \label{Pr:CL1}
 Let $(A_2,\dgal{-,-})$ be a mixed double algebra of weight $(\lambda,\rho)$ such that \eqref{Eq:CL1vv} and \eqref{Eq:CL1mix} hold.
 It is Poisson in the following situations:
\begin{enumerate}
 \item $\rho=-\lambda$ and \eqref{Eq:CL1mix} reads for some $\alpha,\beta\in \{0,\lambda\}$
\begin{equation} \label{Eq:CL1-1}
\dgal{v,w}=\alpha  \, 1 \otimes vw - \beta \, wv \otimes 1 ,  \quad
\dgal{w,v}= (-\lambda + \beta) \, 1 \otimes wv + (\lambda - \alpha) \, vw \otimes 1 ;
\end{equation}
 \item $\rho=+\lambda$ and \eqref{Eq:CL1mix} reads for some $\tilde \alpha,\tilde \beta\in \{0,\lambda\}$
\begin{equation} \label{Eq:CL1-2}
\dgal{v,w}= \tilde \alpha \, v \otimes w - \tilde \beta \, w \otimes v,  \quad
\dgal{w,v}=(-\lambda+\tilde \beta) \, v \otimes w +  (\lambda-\tilde \alpha)\, w \otimes v .
\end{equation}
\end{enumerate}
\end{prop}
\begin{proof}
Note that \eqref{Eq:CL1-1} and \eqref{Eq:CL1-2} correspond respectively to cases \emph{(1)} and \emph{(2)} in Lemma \ref{Lem:CL1}.
Let us verify that \eqref{Eq:DJ2c} is satisfied in those two cases, but that case \emph{(3)} must be discarded.

The right-hand side of \eqref{Eq:DJ2c} identically vanishes as $\dgal{v,v}=0$. The left-hand side of \eqref{Eq:DJ2c} becomes by \eqref{Eq:dbr3}, \eqref{Eq:CL1vv} and \eqref{Eq:CL1vw}:
\begin{equation}  \label{Eq:CL1-cond3}
 \begin{aligned}
\DJac(v,w,v)=& -\frac{(\lambda+\rho+\gamma_1)\gamma_1}{4} v\otimes w\otimes v
- \frac{(\lambda+\rho+\gamma_1)\gamma_3}{4} 1\otimes vw\otimes v \\
&- \frac{(\lambda-\rho+\gamma_3)\gamma_1}{4} v^2\otimes w\otimes 1
- \frac{(\lambda-\rho+\gamma_3)\gamma_3}{4} v\otimes vw\otimes 1 \\
&+ \frac{(\lambda+\rho+\gamma_2)\gamma_2}{4} v\otimes v\otimes w
+ \frac{(\lambda+\rho+\gamma_2)\gamma_4}{4} v^2 \otimes 1\otimes w \\
&+ \frac{(\lambda-\rho+\gamma_4)\gamma_2}{4} 1\otimes v\otimes wv
+ \frac{(\lambda-\rho+\gamma_4)\gamma_4}{4} v\otimes 1\otimes wv
\,.
 \end{aligned}
\end{equation}

If $\rho=-\lambda$, vanishing of \eqref{Eq:CL1-cond3} amounts to:
\begin{align*}
 \gamma_1^2=0, \ \gamma_2^2=0,& \ \gamma_1\gamma_3=0, \ \gamma_2\gamma_4=0, \\
 \gamma_3(\gamma_3+2\lambda)=0, \ \gamma_4(\gamma_4+2\lambda)=0,& \ \gamma_1(\gamma_3+2\lambda)=0, \ \gamma_2(\gamma_4+2\lambda)=0,
\end{align*}
This is satisfied when $\gamma_1=\gamma_2=0$ and $\gamma_3,\gamma_4 \in \{0,-2\lambda\}$ (i.e. case \emph{(1)} in Lemma \ref{Lem:CL1})
but this fails for $\Gamma=(0,-2\lambda,-2\lambda,0),(-2\lambda,0,0,-2\lambda)$ (i.e. case \emph{(3)} in Lemma \ref{Lem:CL1}).

If $\rho=+\lambda$, vanishing of \eqref{Eq:CL1-cond3} amounts to:
\begin{align*}
 \gamma_3^2=0, \ \gamma_4^2=0,& \ \gamma_1\gamma_3=0, \ \gamma_2\gamma_4=0, \\
 \gamma_1(\gamma_1+2\lambda)=0, \ \gamma_2(\gamma_2+2\lambda)=0,& \ \gamma_3(\gamma_1+2\lambda)=0, \ \gamma_4(\gamma_2+2\lambda)=0,
\end{align*}
This is satisfied when $\gamma_3=\gamma_4=0$ and $\gamma_1,\gamma_2 \in \{0,-2\lambda\}$ (i.e. case \emph{(2)} in Lemma \ref{Lem:CL1})
but this fails for $\Gamma=(0,-2\lambda,-2\lambda,0),(-2\lambda,0,0,-2\lambda)$ (i.e. case \emph{(3)} in Lemma \ref{Lem:CL1}).

We leave to the reader the standard but tedious task of checking that \eqref{Eq:DJ2d}--\eqref{Eq:DJ2f} also hold for cases \emph{(1)} and \emph{(2)} from Lemma \ref{Lem:CL1}.
\end{proof}

\begin{cor}
 The $8$ operations considered in Proposition \ref{Pr:CL1} define modified double Poisson brackets.
\end{cor}
\begin{proof}
 This follows from Theorem \ref{Thm:MAIN}.
\end{proof}

\begin{rem}
The $8$ operations considered in Proposition \ref{Pr:CL1} are members of a (conjectural) classification by Arthamonov \cite{Art24} of $12$ modified double Poisson brackets on $\K\langle v,w\rangle$ stable under the $(\K^\times)^2$-action by automorphisms
 \begin{equation}
  (\zeta_1,\zeta_2) \cdot (v,w) = (\zeta_1 v, \zeta_2 w)\,, \quad \zeta_1,\zeta_2 \in \K^\times\,.
 \end{equation}
The remaining $4$ structures defined by Arthamonov are mixed double algebras, hence they define mixed double brackets by Corollary~\ref{Cor:skew}.
However, they do not satisfy our double Jacobi identity \eqref{Eq:DJac-v}, hence we can not prove that the Jacobi identity \eqref{Eq:Jac} always hold for these modified double brackets.
\end{rem}

\subsection{Case \texorpdfstring{$d\geq 3$}{d>=3}}

For $d\geq 3$, fix $\lambda_{1},\ldots,\lambda_{d}\in \K$ and $A_d=\K\langle v_1,\ldots,v_d\rangle$.
\begin{lem} \label{Lem:CL3}
Let $(A_d,\dgal{-,-})$ be a mixed double algebra of weight $(\lambda_{1},\ldots,\lambda_{d})$.
If
\begin{equation} \label{Eq:CL3}
 \begin{aligned}
\dgal{v_i,v_i}&=0, \qquad \dgal{v_j,v_j}=0\,,  \\
\dgal{v_i,v_j}&=-\frac{\gamma_1}{2} \, v_i \otimes v_j + \frac{\gamma_2}{2} \, v_j \otimes v_i
- \frac{\gamma_3}{2} \, 1 \otimes v_i v_j + \frac{\gamma_4}{2} \, v_j v_i \otimes 1\,, \\
\dgal{v_j,v_i}&=-\frac{\lambda_{i} + \lambda_{j} +\gamma_2}{2} \, v_i \otimes v_j
+ \frac{\lambda_{i} + \lambda_{j} +\gamma_1}{2}\, v_j \otimes v_i  \\
& \quad - \frac{\lambda_{i} - \lambda_{j} +\gamma_4}{2} \, 1 \otimes v_j v_i + \frac{\lambda_{i} - \lambda_{j} +\gamma_3}{2} \, v_i v_j \otimes 1\,,
\end{aligned}
\end{equation}
for some distinct $i,j\in \{1,\ldots,d\}$ and $\Gamma=(\gamma_1,\ldots,\gamma_4) \in \K^4$,
then $\lambda_{i}^2-\lambda_{j}^2=0$, and we are in one of the following two situations:
\begin{enumerate}
 \item  $\gamma_1=\gamma_2=0$, $\gamma_3,\gamma_4 \in \{0,-2\lambda_{i}\}$ and $\lambda_{j}=-\lambda_{i}$;
 \item $\gamma_3=\gamma_4=0$, $\gamma_1,\gamma_2 \in \{0,-2\lambda_{i}\}$ and $\lambda_{j}=\lambda_{i}$.
\end{enumerate}
\end{lem}
\begin{proof}
This is a reformulation of Proposition \ref{Pr:CL1} with $v:=v_i$ and $w:=v_j$.
\end{proof}

If we assume that \eqref{Eq:CL3} holds for all $1\leq i<j\leq d$, we are either
in the case of a double Poisson bracket (when $\lambda_{j}=0$ for some $j$, hence for all), or the weight is of the form
\begin{equation*}
 \lambda\, (1,\epsilon_2,\ldots,\epsilon_d)\,, \quad \lambda\in \K^\times, \quad \epsilon_2,\ldots,\epsilon_d =\pm 1.
\end{equation*}
Combining Observations (3) and (4) after Definition \ref{Def:MixDA}, up to rescaling and permutations of generators, the weight must be of the form
$\mathbf{1}_{\delta,d}$ \eqref{Eq:1delD} with $1 \leq \delta \leq \lfloor d/2\rfloor +1$.
For $d=3$, we therefore have $2$ distinct cases to analyze: the weights $(1,1,1)$ and $(1,1,-1)$.

\subsubsection*{$d=3$, weight $(1,1,1)$}

\begin{prop}  \label{Pr:CL3a}
 Let $(A_3,\dgal{-,-})$ be a mixed double algebra of weight $(1,1,1)$ such that \eqref{Eq:CL3} holds for any $1\leq i<j\leq 3$.
Then it is Poisson when it is given by
\begin{equation} \label{Eq:CL3a}
 \begin{aligned}
  &\dgal{v_1,v_1}=0, \quad \dgal{v_2,v_2}=0, \quad \dgal{v_3,v_3}=0, \\
&\dgal{v_1,v_2}= \tilde \alpha_3 \, v_1 \otimes v_2 - \tilde \beta_3 \, v_2 \otimes v_1,  \\
&\dgal{v_2,v_1}=(-1+\tilde \beta_3) \, v_1 \otimes v_2 +  (1-\tilde \alpha_3)\, v_2 \otimes v_1\,, \\
&\dgal{v_1,v_3}= \tilde \alpha_2 \, v_1 \otimes v_3 - \tilde \beta_2 \, v_3 \otimes v_1,  \\
&\dgal{v_3,v_1}=(-1+\tilde \beta_2) \, v_1 \otimes v_3 +  (1-\tilde \alpha_2)\, v_3 \otimes v_1\,, \\
&\dgal{v_2,v_3}= \tilde \alpha_1 \, v_2 \otimes v_3 - \tilde \beta_1 \, v_3 \otimes v_2,  \\
&\dgal{v_3,v_2}=(-1+\tilde \beta_1) \, v_2 \otimes v_3 +  (1-\tilde \alpha_1)\, v_3 \otimes v_2\,,
 \end{aligned}
\end{equation}
for $\tilde\alpha_i,\tilde\beta_i \in \{0,1\}$ subject to the following $2$ conditions:
\begin{subequations}
 \begin{align}
  \tilde\alpha_1 \tilde\alpha_2 + \tilde\alpha_2 \tilde\alpha_3 - \tilde\alpha_1 \tilde\alpha_3 - \tilde\alpha_2 &=0\,, \label{Eq:talph} \\
  \tilde\beta_1 \tilde\beta_2 + \tilde\beta_2 \tilde\beta_3 - \tilde\beta_1 \tilde\beta_3 - \tilde\beta_2 &=0\,. \label{Eq:tbeta}
 \end{align}
\end{subequations}
\end{prop}

\begin{rem}
As we should expect from the second observation made after Definition \ref{Def:MixDA}, the conditions \eqref{Eq:talph}--\eqref{Eq:tbeta} are invariant under permutations of the $3$ generators $v_1,v_2$ and $v_3$.
Indeed, swapping $v_1$ with $v_2$ or $v_2$ with $v_3$ (which generate any permutation) amounts to changing constants according to
\begin{align*}
 (\tilde\alpha_1,\tilde\alpha_2,\tilde\alpha_3) \ &\mapsto \ (\tilde\alpha_2,\tilde\alpha_1,1-\tilde\alpha_3) \quad \text{ under }
 v_1 \leftrightarrow v_2, \\
 (\tilde\alpha_1,\tilde\alpha_2,\tilde\alpha_3)\ &\mapsto \ (1-\tilde\alpha_1,\tilde\alpha_3,\tilde\alpha_2) \quad \text{ under }
 v_2 \leftrightarrow v_3;
\end{align*}
these are transformations preserving \eqref{Eq:talph}. (The same holds for \eqref{Eq:tbeta} if one uses $\tilde\beta_i$ in place of $\tilde\alpha_i$).
In particular, the triples $(\tilde\alpha_1,\tilde\alpha_2,\tilde\alpha_3)$ and $(\tilde\beta_1,\tilde\beta_2,\tilde\beta_3)$ satisfying \eqref{Eq:talph}--\eqref{Eq:tbeta} can be given explicitly as in \eqref{Eq:TripCL3a}.
\end{rem}

\begin{proof}[Proof of Proposition~\ref{Pr:CL3a}]
By the second case in Proposition \ref{Pr:CL1}, \eqref{Eq:DJac-v} is satisfied when $i,j,k$ range over a subset of $2$ indices provided that \eqref{Eq:CL3a} holds for any $\tilde\alpha_i,\tilde\beta_i \in \{0,1\}$.
It remains to check that \eqref{Eq:DJac-v} (with all $\lambda_{i}=1$) holds under the conditions \eqref{Eq:talph}--\eqref{Eq:tbeta} when the $3$ indices are distinct; there are $6$ cases.

Let us compute \eqref{Eq:DJac-v} for $(i,j,k)=(1,2,3)$, the other cases being completely analogous.
To write the left-hand side, we use \eqref{Eq:dbr3} and \eqref{Eq:CL3a}:
\begin{align*}
 \dgal{v_1,\dgal{v_2,v_3}}_L=& \
\tilde{\alpha}_1 \tilde{\alpha}_3 \ v_1 \otimes v_2 \otimes v_3
-\tilde{\alpha}_1 \tilde{\beta}_3 \ v_2 \otimes v_1 \otimes v_3 \\
&-\tilde{\alpha}_2 \tilde{\beta}_1 \ v_1 \otimes v_3 \otimes v_2
+ \tilde{\beta}_1 \tilde{\beta}_2 \ v_3 \otimes v_1 \otimes v_2 \,, \\
- \dgal{v_2,\dgal{v_1,v_3}}_R=&
-\tilde{\alpha}_1 \tilde{\alpha}_2 \ v_1 \otimes v_2 \otimes v_3
+\tilde{\alpha}_2 \tilde{\beta}_1 \ v_1 \otimes v_3 \otimes v_2 \\
&+\tilde{\beta}_2 (-1+\tilde{\beta}_3) \ v_3 \otimes v_1 \otimes v_2
+\tilde{\beta}_2 (1-\tilde{\alpha}_3) \ v_3 \otimes v_2 \otimes v_1 \,,  \\
- \dgal{\dgal{v_1,v_2},v_3}_L=& \
- \tilde{\alpha}_2 \tilde{\alpha}_3 \ v_1 \otimes v_2 \otimes v_3
+ \tilde{\alpha}_3 \tilde{\beta}_2 \ v_3 \otimes v_2 \otimes v_1 \\
&+\tilde{\alpha}_1 \tilde{\beta}_3 \ v_2 \otimes v_1 \otimes v_3
- \tilde{\beta}_1 \tilde{\beta}_3 \ v_3 \otimes v_1 \otimes v_2\,,
\end{align*}
which yield
\begin{equation}
 \begin{aligned} \label{Eq:CL3a-pf1}
\DJac(v_1,v_2,v_3)=&
-(\tilde\alpha_1 \tilde\alpha_2 + \tilde\alpha_2 \tilde\alpha_3 - \tilde\alpha_1 \tilde\alpha_3) \ v_1 \otimes v_2 \otimes v_3
+ \tilde{\beta}_2 \ v_3 \otimes v_2 \otimes v_1 \\
&+(\tilde\beta_1 \tilde\beta_2 + \tilde\beta_2 \tilde\beta_3 - \tilde\beta_1 \tilde\beta_3 - \tilde\beta_2) \ v_3 \otimes v_1 \otimes v_2\,.
 \end{aligned}
\end{equation}
Meanwhile, the right-hand side reads:
\begin{equation} \label{Eq:CL3a-pf2}
 \begin{aligned}
-v_2 \otimes_1 \dgal{v_1,v_3}=&
-\tilde\alpha_2  \ v_1 \otimes v_2 \otimes v_3
+ \tilde{\beta}_2 \ v_3 \otimes v_2 \otimes v_1 \,.
 \end{aligned}
\end{equation}
The expressions \eqref{Eq:CL3a-pf1} and \eqref{Eq:CL3a-pf2} coincide precisely when \eqref{Eq:talph}--\eqref{Eq:tbeta} hold.
\end{proof}

\begin{cor} \label{Cor:CL3a}
 Under the conditions \eqref{Eq:talph}--\eqref{Eq:tbeta}, the operation \eqref{Eq:CL3a} defines:
\begin{enumerate}
\item a $1$-double Lie algebra structure on $V:=\K v_1 \oplus \K v_2 \oplus \K v_3$;
 \item a modified double Poisson bracket on $A_3$.
\end{enumerate}
\end{cor}
\begin{proof}
 The first part follows from Definition~\ref{Def:GG} and Proposition~\ref{Pr:CL3a} (cf. Example~\ref{Exmp:GG}).

The second part follows  from Proposition~\ref{Pr:CL3a} and Theorem~\ref{Thm:MAIN}.
\end{proof}

\subsubsection*{$d=3$, weight $(1,1,-1)$}

\begin{prop}  \label{Pr:CL3b}
 Let $(A_3,\dgal{-,-})$ be a mixed double algebra of weight $(1,1,-1)$ such that \eqref{Eq:CL3} holds for any $1\leq i<j\leq 3$.
Then it is Poisson when it is given by \eqref{Eq:CL3b}
for $\alpha_1, \alpha_2,\tilde\alpha_3,\beta_1, \beta_2,\tilde\beta_3 \in \{0,1\}$ subject to the following $2$ conditions:
\begin{subequations}
 \begin{align}
  \alpha_1 \alpha_2 + \alpha_2 \tilde\beta_3 - \alpha_1 \tilde\beta_3 - \alpha_2 &=0\,, \label{Eq:t3b1} \\
  \beta_1 \beta_2 + \beta_2 \tilde\alpha_3 - \beta_1 \tilde\alpha_3 - \beta_2 &=0\,. \label{Eq:t3b2}
 \end{align}
\end{subequations}
\end{prop}
\begin{rem}
Conditions \eqref{Eq:t3b1}--\eqref{Eq:t3b2} are invariant under the permutation $v_1\leftrightarrow v_2$ which preserves the weight $(1,1,-1)$.
 In particular, we deduce from \eqref{Eq:t3b1}--\eqref{Eq:t3b2} that the triples $(\alpha_1,\alpha_2,\tilde \beta_3)$ and $(\beta_1,\beta_2,\tilde \alpha_3)$ can only take the $6$ values collected in \eqref{Eq:TripCL3a}.
\end{rem}

\begin{proof}[Proof of Proposition~\ref{Pr:CL3b}]
This is an explicit computation similar to Proposition \ref{Pr:CL3a}.
For the reader's convenience, let us nevertheless check \eqref{Eq:DJac-v} for $(i,j,k)=(1,3,2)$
(recall the weight $(1,1,-1)$).
To write the left-hand side, we use \eqref{Eq:dbr3}, \eqref{Eq:Lei} and \eqref{Eq:CL3b}:
\begin{align*}
 \dgal{v_1,\dgal{v_3,v_2}}_L=& \
(1-\alpha_1) \tilde{\alpha}_3 \ v_1 \otimes v_2 v_3 \otimes 1
-(1-\alpha_1) \tilde{\beta}_3 \ v_2 \otimes v_1 v_3 \otimes 1 \\
&+(1-\alpha_1) \alpha_2 \ v_2 \otimes v_1 v_3 \otimes 1
- (1-\alpha_1) \beta_2 \ v_2 v_3v_1 \otimes 1 \otimes 1 \,, \\
- \dgal{v_3,\dgal{v_1,v_2}}_R=&
-\tilde{\alpha}_3 (-1+\beta_1) \ v_1 \otimes 1 \otimes v_3 v_2
-\tilde{\alpha}_3 (1-\alpha_1) \ v_1 \otimes v_2 v_3 \otimes 1 \\
&+\tilde{\beta}_3 (-1+\beta_2) \ v_2 \otimes 1 \otimes v_3 v_1
+\tilde{\beta}_3 (1-\alpha_2) \ v_2 \otimes v_1 v_3 \otimes 1 \,,  \\
- \dgal{\dgal{v_1,v_2},v_3}_L=& \
\beta_2 \tilde{\alpha}_3 \ v_1 \otimes 1 \otimes v_3 v_2
- \beta_2 \tilde{\beta}_3 \ v_2 \otimes 1 \otimes v_3 v_1 \\
&+\beta_2 (-1+\beta_1) \ v_1 \otimes 1 \otimes v_3 v_2
+ \beta_2 (1-\alpha_1)\ v_2 v_3 v_1 \otimes 1 \otimes 1\,,
\end{align*}
which yield
\begin{equation}
 \begin{aligned} \label{Eq:CL3b-pf1}
\DJac(v_1,v_2,v_3)=&
-(\alpha_1 \alpha_2 + \alpha_2 \tilde\beta_3 - \alpha_1 \tilde\beta_3 - \alpha_2) \ v_2 \otimes v_1 v_3 \otimes 1
- \tilde{\beta}_3 \ v_2 \otimes 1 \otimes v_3 v_1 \\
&+(\beta_1 \beta_2 + \beta_2 \tilde\alpha_3 - \beta_1 \tilde\alpha_3 - \beta_2 + \tilde \alpha_3) \ v_1 \otimes 1 \otimes v_3 v_2\,.
 \end{aligned}
\end{equation}
Meanwhile, the right-hand side reads:
\begin{equation} \label{Eq:CL3b-pf2}
 \begin{aligned}
1 \otimes_1 (v_3 \ast \dgal{v_1,v_2})=&
\tilde \alpha_3 \ v_1 \otimes 1 \otimes v_3 v_2
- \tilde \beta_3 \ v_2 \otimes 1 \otimes v_3 v_1 \,.
 \end{aligned}
\end{equation}
The expressions \eqref{Eq:CL3b-pf1} and \eqref{Eq:CL3b-pf2} coincide precisely when \eqref{Eq:t3b1}--\eqref{Eq:t3b2} hold.
\end{proof}

\begin{cor} \label{Cor:CL3b}
 Under the conditions \eqref{Eq:t3b1}--\eqref{Eq:t3b2}, the operation \eqref{Eq:CL3b} defines a modified double Poisson bracket on $A_3$.
\end{cor}

Combining Propositions \ref{Pr:CL3a} and \ref{Pr:CL3b}, we can construct many new modified double Poisson brackets.
This is how we found Propositions \ref{Pr:CLd} and \ref{Pr:CLd-2}, which are proved now.

\subsubsection{Proof of Proposition \ref{Pr:CLd}} \label{ss:PfCLd}

By checking \eqref{Eq:intr1} for any $1\leq i \leq j \leq d$, it is clear that \eqref{Eq:CL4-1} defines a mixed double algebra of weight $\mathbf{1}_{\delta,d}$. Hence it remains to verify \eqref{Eq:DJac-v}.

Pick $1\leq a<b<c\leq d$.
If $c\leq \delta$, the mixed double brackets involving $v_a,v_b,v_c$ correspond to taking all constants equal to $+1$ in \eqref{Eq:CL3a} with $(v_a,v_b,v_c)$ relabelled as $(v_1,v_2,v_3)$. Thus \eqref{Eq:DJac-v} is satisfied by Proposition \ref{Pr:CL3a} whenever $v_i,v_j,v_k \in \{v_a,v_b,v_c\}$.

Similarly, if $b\leq \delta$ and $c>\delta$, the mixed double brackets involving $v_a,v_b,v_c$ are of weight $(1,1,-1)$ and correspond to taking all constants equal to $+1$ in \eqref{Eq:CL3b} (with $(v_a,v_b,v_c)$ relabelled as $(v_1,v_2,v_3)$), hence \eqref{Eq:DJac-v} is satisfied on these generators by Proposition \ref{Pr:CL3b}.

If $a\leq \delta$ and $b>\delta$, the mixed double brackets involving $v_a,v_b,v_c$ are of weight $(1,-1,-1)$.
Up to multiplying $\dgal{-,-}$ by $-1$, the weight is $(-1,1,1)$ and they correspond to
\begin{equation*}
 (\alpha_1,\alpha_2,\tilde \alpha_3)=(0,0,1), \quad (\beta_1,\beta_2,\tilde \beta_3)=(0,0,1),
\end{equation*}
in \eqref{Eq:CL3b} after relabelling $(v_a,v_b,v_c)$ as $(v_3,v_1,v_2)$. Hence \eqref{Eq:DJac-v} is satisfied on these generators by Proposition \ref{Pr:CL3b}.

Finally, if $a>\delta$, the mixed double brackets involving $v_a,v_b$ and $v_c$ are of weight $(-1,-1,-1)$.
Up to multiplying $\dgal{-,-}$ by $-1$, these are of weight $(1,1,1)$ and correspond to taking all constants equal to $+1$ in \eqref{Eq:CL3a} (with $(v_a,v_b,v_c)$ relabelled as $(v_1,v_2,v_3)$), hence \eqref{Eq:DJac-v} is satisfied on these generators by Proposition \ref{Pr:CL3a}.

The last part follows from Theorem \ref{Thm:MAIN}. \qed

\subsubsection{Proof of Proposition \ref{Pr:CLd-2}}  \label{ss:PfCLd-2}

This is similar to the proof of Proposition \ref{Pr:CLd}. The only changes are as follows:
\begin{itemize}
 \item If $c\leq \delta$, we use the constants $(\tilde \alpha_1,\tilde \alpha_2,\tilde \alpha_3)=(1,1,1)$ and
 $(\tilde \beta_1,\tilde \beta_2,\tilde \beta_3)=(0,0,0)$. If $a>\delta$, we take the same constants.
 \item If $b\leq \delta$ and $c>\delta$, we use the following constants $(\alpha_1,\alpha_2,\tilde \beta_3)=(0,0,0)$ and
 $(\beta_1,\beta_2,\tilde \alpha_3)=(1,1,1)$.  If $a\leq \delta$ and $b>\delta$, we use the same constants (recalling that we need a different relabelling in that case). \qed
\end{itemize}

\subsection{Proof of the first instance of Arthamonov's conjecture (Theorem \ref{Thm:GGconj})}  \label{ss:ArtConj}

If we look at the operation $\dgal{-,-}^{I\!I}$ defined in \eqref{Eq:MDBII} on $\K\langle x_1,x_2,x_3 \rangle$, we see that
 \begin{align*}
&\dgal{x_1,x_2}^{I\!I}+(\dgal{x_2,x_1}^{I\!I})^\circ=-(x_1 \otimes x_2- x_2 \otimes x_1), \\
&\dgal{x_2,x_3}^{I\!I}+(\dgal{x_3,x_2}^{I\!I})^\circ=-(x_2 \otimes x_3 - x_3\otimes x_2), \\
&\dgal{x_3,x_1}^{I\!I}+(\dgal{x_1,x_3}^{I\!I})^\circ= - (x_3 \otimes x_1-x_1 \otimes x_3).
 \end{align*}
 This defines a mixed double algebra of weight $(-1,-1,-1)$.

 Up to multiplying $\dgal{-,-}^{I\!I}$ by $-1$ and setting $v_i:=x_i$, we get a mixed double algebra structure of weight $(1,1,1)$ that reads:
\begin{equation} \label{Eq:MDBIIbis}
 \begin{aligned}
&\dgal{v_1,v_2}=v_1 \otimes v_2,  \quad &&\dgal{v_2,v_1}=-v_1 \otimes v_2, \\
&\dgal{v_2,v_3}=-v_3\otimes v_2,  \quad &&\dgal{v_3,v_2}=v_3 \otimes v_2, \\
&\dgal{v_3,v_1}=-v_1 \otimes v_3 + v_3 \otimes v_1,&&
 \end{aligned}
\end{equation}
where zero brackets are omitted. If we take in \eqref{Eq:CL3a} the constants
\begin{equation*}
 (\tilde\alpha_1,\tilde\alpha_2,\tilde\alpha_3)= (0,0,1), \quad
 (\tilde\beta_1,\tilde\beta_2,\tilde\beta_3) = (1,0,0),
\end{equation*}
we reproduce \eqref{Eq:MDBIIbis}. By Proposition~\ref{Pr:CL3a}, this is a mixed double Poisson algebra, hence $\dgal{-,-}^{I\!I}$ is a modified double Poisson bracket by Corollary \ref{Cor:CL3a}.


\subsection*{Acknowledgments}
I am indebted to V.~Gubarev and S.~Arthamonov for useful correspondence, and to the referee for numerous suggestions.
This work was completed at IMB which is supported by the EIPHI Graduate School (contract ANR-17-EURE-0002).


\EditInfo{June 18, 2024}{September 24, 2024}{David Towers and Ivan Kaygorodov}

\end{document}